\documentclass[11pt]{amsart}
\usepackage{amssymb,amscd,  latexsym, graphicx, mathrsfs, enumerate}
\usepackage{color}

\newcommand{\nc}{\newcommand}
\numberwithin{equation}{section}
\newtheorem{theorem}{Theorem}[section]
\newtheorem{prop}[theorem]{Proposition}
\newtheorem{importnota}[theorem]{Important Notation}
\newtheorem{prblm}[theorem]{Problem}
\newtheorem{notation}[theorem]{Notation}
\newtheorem{caution}[theorem]{Caution}
\newtheorem{remark}[theorem]{Remark}
\newtheorem{lemma}[theorem]{Lemma}
\newtheorem{construction}[theorem]{Construction}
\newtheorem{corollary}[theorem]{Corollary}
\newtheorem{example}[theorem]{Example}
\newtheorem{conclusion}[theorem]{Conclusion}
\newtheorem{triviality}[theorem]{Triviality}
\newtheorem{proto}[theorem]{Prototype Quasifibration}
\newtheorem{cauex}[theorem]{Cautionary Example}
\newtheorem{propositiondef}[theorem]{Proposition-Definition}
\newtheorem{subth}{Nuisance}[theorem]
\newtheorem{ssubth}{ }[subth]
\newtheorem{conjecture}[theorem]{Conjecture}
\newtheorem{sidest}[theorem]{Side Story}
\newtheorem{miniexample}[theorem]{Example}
\theoremstyle{definition}
\newtheorem{defin}[theorem]{Definition}

\nc\tri[1]{\begin{triviality}}
\nc\side[1]{\begin{sidest}}
\nc\conj[1]{\begin{conjecture}}
\nc\prodef[1]{\begin{propositiondef}}
\nc\prt[1]{\begin{proto}}
\nc\lem[1]{\begin{lemma}}
\nc\sblm[1]{\begin{sublemma}}
\nc\pro[1]{\begin{prop}}
\nc\thm[1]{\begin{theorem}}
\nc\cor[1]{\begin{corollary}}
\nc\dfn[1]{\begin{defin}}
\nc\sthm[1]{\begin{subth}}
\nc\exm[1]{\begin{example}}
\nc\miniexm[1]{\begin{miniexample}}
\nc\plm[1]{\begin{prblm}}
\nc\rmk[1]{\begin{remark}}
\nc\subrmk[1]{\begin{subremark}}
\nc\ntn[1]{\begin{notation}}
\nc\cau[1]{\begin{caution}}
\nc\imn[1]{\begin{importnota}}
\nc\cax[1]{\begin{cauex}}
\nc\con[1]{\begin{construction}}
\nc\ssthm[1]{\begin{ssubth}}
\nc\cnc[1]{\begin{conclusion}}
\nc\elem{\end{lemma}}
\nc\esblm{\end{sublemma}}
\nc\eside{\end{sidest}}
\nc\econj{\end{conjecture}}
\nc\eprodef{\end{propositiondef}}
\nc\eprt{\end{proto}}
\nc\ethm{\end{theorem}}
\nc\ecor{\end{corollary}}
\nc\edfn{\end{defin}}
\nc\esthm{\end{subth}}
\nc\epro{\end{prop}}
\nc\etri{\end{triviality}}
\nc\eexm{\end{example}}
\nc\eminiexm{\end{miniexample}}
\nc\ermk{\end{remark}}
\nc\subermk{\end{subremark}}
\nc\eplm{\end{prblm}}
\nc\ecau{\end{caution}}
\nc\ecax{\end{cauex}}
\nc\eimn{\end{importnota}}
\nc\entn{\end{notation}}
\nc\econ{\end{construction}}
\nc\ecnc{\end{conclusion}}
\nc\essthm{\end{ssubth}}

\newcommand{\C}{\mathbb{C}}
\newcommand{\R}{\mathbb{R}}
\newcommand{\Q}{\mathbb{Q}}

\newcommand{\Z}{\mathbb{Z}}

\newcommand{\A}{\mathbb{A}}

\newcommand{\I}{\overline{I}}

\newcommand{\cS}{\mathcal{S}}

\newcommand{\ve}{\varepsilon}
\newcommand{\diag}{{\rm diag}}
\newcommand{\G}{\Gamma}

\newcommand{\ds}{\displaystyle}

\newcommand{\f}{\bold{f}}
\newcommand{\Ind}{{\rm Ind}}
\newcommand{\bw}{\textbf{w}}

\newcommand{\Sym}{{\rm Sym}}

\newcommand{\oH}{\overline{H}}

\newcommand{\lra}{\longrightarrow}

\newcommand{\ba}{\backslash}
\newcommand{\sgn}{{\rm sgn}}
\newcommand{\GL}{{\rm GL}}

\newcommand{\SL}{{\rm SL}}
\newcommand{\SO}{{\rm SO}}
\newcommand{\OO}{{\rm O}}
\newcommand{\bs}{\backslash}
\newcommand{\disc}{{\rm disc}}

\renewcommand{\Bbb}{\mathbb}
\newcommand{\W}{\mathcal{W}}

\title[On explicit Fourier expansions of theta lifts to $\SO(3,n+1)$]
{On explicit Fourier expansions of theta lifts to $\SO(3,n+1)$ arising from 
elliptic newforms of level one}
\author{Henry H. Kim and Takuya Yamauchi}
\date{\today}
\thanks{The first author is partially supported by NSERC grant \#482564. 
}
\address{Henry H. Kim \\
Department of mathematics \\
 University of Toronto \\
Toronto, Ontario M5S 2E4, CANADA \\
and Korea Institute for Advanced Study, Seoul, KOREA}
\email{henrykim@math.toronto.edu}

\address{Takuya Yamauchi \\
Mathematical Inst. Tohoku Univ.\\
 6-3,Aoba, Aramaki, Aoba-Ku, Sendai 980-8578, JAPAN}
\email{takuya.yamauchi.c3@tohoku.ac.jp}


\begin{document}
\maketitle
\begin{abstract}
Using degenerate Whittaker functions and explicit computations of Eisenstein series, 
we obtain explicit formulas for the Fourier expansions of theta lifts to the special orthogonal group 
$G=\SO(3,n+1)$ over $\Q$, where $n\ge 3$ and $G$ splits at all finite places. 
The theta lifts in question are Hecke eigen, non-cuspidal, square-integrable automorphic forms of weight $l$ 
($l\ge n+2$, even), arising from elliptic newforms for $\SL_2(\Z)$ of weight 
$l-\frac{n-2}{2}$ when $n$ is even and $2l-n+1$ when $n$ is odd.
\end{abstract}

\tableofcontents



\section{Introduction}\label{intro}

Constructing automorphic forms on reductive groups is an important subject 
in both the theory of automorphic representations and number theory. 
There are several important works on automorphic forms by trace formulas \cite{Arthur}, 
theta lifts \cite{Rallis} (see \cite{MS} as a related work), the descent method \cite{GJS}, 
and converse theorems \cite{GRS}. 
While these approaches provide various ways to construct or characterize automorphic forms, 
a deeper understanding of their properties often relies on the analysis of their Fourier expansions. 
In particular, the theory of Fourier expansion plays a central role in studying automorphic forms.

In this paper, we study explicit formulas for the Fourier expansion of theta lifts 
from level-one holomorphic elliptic newforms to $\SO(3,n+1)$ 
in terms of degenerate Whittaker functions, 
based on several ideas from Ikeda-type constructions 
(see \cite{Ik01}, \cite{Ik08}, \cite{IY}, \cite{KY}, \cite{KY1}).  
We remark that Miyazaki and Saito have studied an explicit formula of the theta lifts in 
more general setting using the Borcherds method (see \cite[Proposition 5.4]{MS}). 
Our method relies on the local and global representation theories and thus, 
they have different natures.   

Let $\A=\A_\Q$ be the ring of adeles of $\Q$. 
For an integer $n\ge 2$, let $A$ be a positive definite integral matrix of size $n-2$. 
Assume $A$ is even (namely, any diagonal entries are even integers)  if $n$ is even while $A=\diag(1,A')$ if $n$ is odd where $A'$ is 
a positive definite integral matrix of size $n-3$.  
Let $J_{2,n}={\rm antidiag}(1,1,-A,1,1)$ and 
$J_{3,n+1}={\rm antidiag}(1,1,1,-A,1,1,1)$ which are of size $n+2$ and $n+4$ 
respectively (see Section \ref{pre}). 
Let $G=\SO(3,n+1)$ be the special orthogonal group over $\Q$ associated to 
the symmetric pairing defined by $J_{3,n+1}$. Note that if $n$ is even, $G$ does not have discrete series.
Assume $G$ splits everywhere at finite places.  
By \cite[Section 2.1]{Serre}, it is equivalent to 
$\det(A)=1$ and $n\equiv 2$ mod 8 when $n$ is even, 
$\det(A')=1$ and $n\equiv 3$ mod 8 when $n$ is odd. 
For instance, this condition is satisfied when $n=8a+2$,  
$A=\diag(\overbrace{A_8,\ldots,A_8}^{a})$ or $n=8a+3$, 
$A=\diag(1,\overbrace{A_8,\ldots,A_8}^{a})$
where $A_8$ is the $E_8$ Cartan matrix given as an element of even integral matrix in $M_8(\Z)$. This agrees with the result of \cite{G} since $\frac {n+4}2\equiv 3$ (mod 4) if $n$ is even, and $\frac {n+3}2\equiv 3$ (mod 4) if $n$ is odd. 

Let $P$ be the Siegel parabolic subgroup of $G$ with the Levi decomposition 
$P=MN$ where $M=\{\diag(t,m,t^{-1})\ |\ t\in \GL_1,\ m\in \SO(2,n)\}\simeq \GL_1\times \SO(2,n)$ and $N\simeq \mathbb{G}^{n+2}_a$ where $\SO(2,n)$ is defined by $J_{2,n}$. Let $\nu:P\lra \GL_1$ be the similitude character which is given as 
the natural extension of the character $M\lra \GL_1,\ \diag(t,m,t^{-1})\mapsto t$ and  
let us extend it naturally on $P$. 
Let $V'$ be the quadratic space associated to $J_{2,n}$ with the quadratic map 
$q:V'\lra \mathbb{G}_a$ and we identify $V'$ with $N$. 
We also denote by $(\ast,\ast)_{V'}$ the corresponding symmetric bilinear pairing. 
Let $V''$ be the quadratic space associated to $J_{1,n-1}=
{\rm antidiag}(1,-A,1)$.

Let $l>n+1$ be an even integer. 
Let $f(z)=\ds\sum_{m=1}^\infty a_f(m) m^{\frac{k-1}{2}}q^m$ be a normalized Hecke eigen cusp form of weight $k=\begin{cases}
l-\frac{n-2}{2} & \text{if $n$ is even},\\
2l-n+1 & \text{if $n$ is odd}
\end{cases}$, with respect to $\SL_2(\Z)$. 
For each prime $p$, let $a_f(p)=\alpha_p+\alpha^{-1}_p$ with $\alpha_p\in U(1)=\{z\in \C\ 
|\ |z|=1\}$. Let 
$\pi=\pi_f=\otimes' \pi_p\otimes \pi_\infty $ 
be the unitary cuspidal representation of $\GL_2(\A)$ associated to $f$. For each prime $p$, we have $\pi_p\simeq \pi_p(\mu_p,\mu^{-1}_p)$ 
so that $\mu_p(p)=\alpha_p$. 
Let $\psi=\otimes'_p \psi_p:\Q\bs \A\lra \C^\times$ be the standard additive character. 
When $n$ is even, let $\tau$ be a unique $\psi$-generic unitary 
cuspidal component of $\pi|_{\SL_2(\A)}$. When $n$ is odd, let $\tau$ be 
the Shimura-Waldspurger lift of $\pi$ to $\widetilde{\SL}_2(\A)$.  
Let $\kappa=\left\{\begin{array}{cl}
2 & \text{if $n$ is even,} \\
1 & \text{if $n$ is odd}
\end{array}\right.$. 
Let $\Pi_f=\otimes'_p \Pi_{f,p}\otimes \Pi_{f,\infty}$ be a unique component 
of the restriction to $G$ of the theta lift of $\tau$ to ${\rm O}(3,n+1)$ such that 
$\Pi_{f,\infty}$ contains $\mathcal{H}^{2l}(\C^3)\boxtimes \mathbf{1}$ as a minimal 
$K_\infty$-type (see Proposition \ref{thetaprop} for the notation).  A standard method shows that $\Pi_f$ is a non-zero square integrable, cohomological, non-cuspidal automorphic representation of $G(\A)$. Hence it contributes to the residual spectrum $L^2_{{\rm res}}(G(\Q)\bs G(\A))$.
Further, by \cite{AG} for 
non-archimedean places and \cite{Kobayashi}, \cite[Proposition 2.1]{MS} for the archimedean place, we see that 
$\Pi_{f,p}$ is a unique spherical summand of the degenerate principal series 
${\rm Ind}^{G(\Q_p)}_{P(\Q_p)}\mu^\kappa_p\circ \nu$ 
and $\Pi_{f,\infty}$ is a unique irreducible component of $\Pi^{3,n+1}_{l,0}|_{G(\R)}$ containing $\mathcal{H}^{2l}(\C^3)\boxtimes 
\mathbf{1}$ as a minimal $K_\infty$-type. Here $\Pi^{3,n+1}_{l,0}$ is the unitarizable, irreducible non-tempered, cohomological admissible representation of $\OO(V)(\R)$ 
given in \cite[Proposition 2.1]{MS}. 
By using the degenerate Whittaker functions (cf. \cite{IY},  \cite{KY1}),  for each prime $p<\infty$ and $\eta\in V'(\Q_p)$ 
with $q(\eta)\neq 0$, we can define 
a non-zero functional 
$$w^{\mu_p^\kappa}_\eta\in {\rm Wh}(\Pi_{f,p}):={\rm Hom}_{N(\Q_p)}(\Pi_{f,p},\psi_p((\eta,\ast)_{V'}).$$
Let us fix an intertwining map   
$$\Pi_f\hookrightarrow L^2(G(\Q)\bs G(\A)),\quad \phi\mapsto F_\phi.
$$ 
By multiplicity one for ${\rm Wh}(\Pi_{f,p})$ for each prime $p$ 
and by \cite[p.621, Theorem 3.2.4]{Po} for the archimedean place, we can expand $F_\phi=(F_\phi)_0+(F_\phi)_+$ as follows for 
each distinguished vector $\phi=\otimes'_p \phi_p\otimes \phi^{{\rm Po}}_\infty=\phi_\f\otimes \phi^{{\rm Po}}_\infty$ with the Pollack's section $\phi^{{\rm Po}}_\infty\in 
\Pi_{f,\infty}\subset \Pi^{3,n+1}_{l,0}$ (see Definition \ref{PollackSection} for $\phi^{{\rm Po}}_\infty$). The initial term $(F_\phi)_0$ is the constant term of $F_\phi$ along $N$ and the second term is given by   
$$(F_\phi)_+(g)=\sum_{\eta\in V'(\Q)\atop q(\eta)>0}c_\eta(\phi)w^{\mu^\kappa_\f}_\eta(g_\f\cdot \phi_\f)\mathcal{W}_{2\pi \eta}(g_\infty),\ 
c_\eta(\phi)\in \C,\ g=g_\f g_\infty\in G(\A)$$
where  $w^{\mu^\kappa_\f}_\eta(g_\f\cdot \phi_\f):=\ds\prod_{p<\infty}w^{\mu^\kappa_p}_\eta(g_p\cdot \phi_p)$ for 
$g_\f=(g_p)_p$ and $\mathcal W_{2\pi \eta}$ is the Pollack's spherical function 
associated to the additive character $e^{2\pi \sqrt{-1}(\eta,\ast)_{V'}}$ on $V'(\R)$ 
(see \cite[Section 3]{Po} but we use a slightly different normalization). We will explain later that 
any index $0\neq \eta\in V'(\Q)$ with $q(\eta)=0$ does not appear in the expansion of the second term $(F_\phi)_+$. For each unramified vector $\phi_p$, the value 
 $w^{\mu_p}_\eta(\phi_p)$ is given explicitly by using the Laurent polynomials 
 arising from the Siegel series for the Eisenstein series. 
 
Our main purpose of this paper is to determine the Fourier coefficients $c_\eta(\phi_\f)$ 
using Fourier-Jacobi expansions for the distinguished unramified vector 
$\phi_\f=\otimes'_{p<\infty}\phi_p$.

In \cite[Section 4]{Po}, Pollack defined the Eisenstein series $E_l$ on $G(\A)$ of weight $l$ with respect to 
$G(\widehat{\Z})$,  
 and computed unramified and archimedean part of Fourier coefficients. 
He computed the Siegel series for the unramified case and the remaining case is completed in 
our previous work \cite{KY2}. 
Then, we have the Fourier expansion of the Eisenstein series as follows:
$$E_l(g)=E_0(g)+\sum_{0\neq \eta\in V'(\Q)\atop q(\eta)\ge 0}a_{E_l}(\eta)\mathcal W_{2\pi \eta}(g_\infty),\ 
g=\gamma k_f g_\infty\in G(\A)=G(\Q)(G(\widehat{\Z})\times G(\R)),$$
where $E_0(g)$ is the constant term. We remark that $a_{E_l}(\eta)=0$ unless 
$\eta\in V'(\Z)$.  

When $n$ is even ($n>2$),
for each $\eta\in V'(\Q)$ with $q(\eta)>0$ and each rational prime $p$, there exists a Laurent polynomial $\widetilde{Q}_{\eta,\Phi_p}(X_p)=\widetilde{Q}_{\eta,p}(X_p)\in \Z[X_p,X^{-1}_p]$ 
depending on $\eta$ and the unramified Schwartz function $\Phi_p={\rm char}_{V'(\Z_p)} \in \mathcal{S}(V'(\Q_p))$ satisfying 
$\widetilde{Q}_{\eta,p}(X_p)=\widetilde{Q}_{\eta,p}(X^{-1}_p)$ (see \cite[Theorem 4.2]{KY2}) such that 
$$a_{E_l}(\eta)=C_{l,n}q(\eta)^{\frac{l-\frac{n}{2}}{2}}\ds\prod_{p}
\widetilde{Q}_{\eta,p}(p^{\frac{l-\frac{n}{2}}{2}})$$
where the constant $C_{l,n}$ is given explicitly in \cite[Theorem 4.8]{KY2}. 
We remark that $\widetilde{Q}_{\eta,\Phi_p}(X_p)$ is also defined for any $\eta\in V'(\Q_p)$ 
with $q(\eta)\neq 0$. For any $g_p=\diag(t_p,m_p,t^{-1}_p)n_p k_p \in G(\Q_p)=M(\Q_p)N(\Q_p)G(\Z_p)$, 
it follows from Proposition \ref{propertiesIW}(5) and \cite[Section 4]{KY2} that  
$\widetilde{Q}_{\eta,g_p\cdot\Phi_p}(X_p)=\psi((t_p\cdot \eta\cdot {}^t m^{-1}_p,n_p)_{V'}) \widetilde{Q}_{g^{-1}_p\eta,\Phi_p}(X_p)$. 
Thus, $\widetilde{Q}_{\eta,g_p\cdot\Phi_p}(X_p)$ is also a Laurent polynomial (over $\Q$) satisfying 
$\widetilde{Q}_{\eta,g_p\cdot\Phi_p}(X_p)=\widetilde{Q}_{\eta,g_p\cdot\Phi_p}(X^{-1}_p)$.  

Now we are ready to explain our main results. 
\begin{theorem}[$n$ even ($n>2$)]\label{ITC}  
Let $A_f(\eta)(g_\f)=q(\eta)^{\frac{k-1}{2}}\ds\prod_p \widetilde 
Q_{\eta,g_p\cdot\Phi_p}(\alpha_p)$ for $g_\f=(g_p)_p\in G(\A_\f)$. Then, for $\phi\in \Pi_f$ with 
$\phi_\f$ the unramified distinguished vector and $\phi_\infty=\phi^{{\rm Po}}_\infty$, there exists an 
absolute non-zero constant $C$ such that 
$$F_\phi(g)=(F_\phi)_0+C\sum_{0\neq \eta\in V'(\Bbb Q)\atop q(\eta)> 0} A_f(\eta)(g_\f)
 \mathcal W_{2\pi \eta}(g_\infty),\ g=g_\f g_\infty \in G(\A)
$$
and it is a non-zero Hecke eigen, square integrable automorphic form of weight $l$ with respect to $G(\widehat{\Z})$. 
\end{theorem}

\begin{remark}\label{standLeven} The standard $L$-function of $\Pi_{f}$ is given by
$$L(s,\Pi_{f})=L(s,{\rm Sym}^2\pi_f)\prod_{i=-\frac n2}^\frac n2 \zeta(s+i).
$$
\end{remark}

When $n$ is odd, we need a slight modification. Let $l>n+1$ be an even integer. 
Let $f$ be a normalized Hecke eigen cusp form of weight $k:=2l-n+1$ with respect to $\SL_2(\Bbb Z)$.
Let $f(z)=\ds\sum_{m=1}^\infty a_f(m)m^{\frac{k-1}{2}} q^m$. Let $a_f(p)=\alpha_p+\alpha_p^{-1}$. 
Let $h(z)=\ds\sum_{m>0\atop \text{$(-1)^{\frac {k}2} m\equiv 0,1$ (mod 4)}} c(m)q^m\in S_{l-\frac {n}2+1}(\Gamma_0(4))$ be the one corresponding to $f$ by the Shimura correspondence. Write $q(\eta)=\frak d_{\eta}\frak f_\eta^2$ such that $\frak d_\eta$ is the absolute discriminant of $\Bbb Q\left(\sqrt{\epsilon q(\eta)}\right)/\Bbb Q$, where $\epsilon=\begin{cases} 1, &\text{if $q(\eta)_1\equiv 1$ (mod 4)}\\ -1, &\text{if $q(\eta)_1\equiv 3$ (mod 4)}\end{cases}$, and $q(\eta)=2^{v_2(\eta)}q(\eta)_1$, $q(\eta)_1$ odd. Let $\chi_\eta$ be the primitive Dirichlet character corresponding to $\Bbb Q\left(\sqrt{\epsilon q(\eta)}\right)/\Bbb Q$. Then for each rational prime $p$, there exists a Laurent polynomial $\widetilde{Q}_{\eta,\Phi_p}(X_p)=\widetilde{Q}_{\eta,p}(X_p)\in \Z[X_p,X^{-1}_p]$ depending on $\eta$ and $\Phi_p={\rm char}_{V'(\Z_p)}$ satisfying 
$\widetilde{Q}_{\eta,p}(X_p)=\widetilde{Q}_{\eta,p}(X^{-1}_p)$  (see \cite[Theorem 4.3, 4.4]{KY2})  such that
$$a_{E_l}(\eta)=
C_{l,n}' L(\tfrac {n+1}2-l,\chi_\eta) \frak f_\eta^{l-\frac {n}2} \prod_p \widetilde 
 Q_{\eta,p}(p^{l-\frac {n}2}),
$$
where $C_{l,n}'$ is given in \cite[Theorem 4.8]{KY2}. 

\begin{theorem}[$n$ odd ($n\geq 3$)]\label{ITC-odd} 
Let $A_f(\eta)(g_\f)=c(\frak d_{\eta}) \frak f_{\eta}^{\frac{k-1}{2}}\ds\prod_{p|q(\eta)} \widetilde Q_{\eta,g_p\cdot \Phi_p}(\alpha_p)$ for $g_\f=(g_p)_p\in G(\A_\f)$. Then, for $\phi\in \Pi_f$ with 
$\phi_\f$ the unramified distinguished vector and $\phi_\infty=\phi^{{\rm Po}}_\infty$, there exists an 
absolute non-zero constant $C$ such that 
$$F_\phi(g)=(F_\phi)_0+C\sum_{0\neq\eta\in V'(\Bbb Q)\atop q(\eta)> 0} A_f(\eta)(g_\f)
\mathcal W_{2\pi \eta}(g_\infty),\ g=g_\f g_\infty \in G(\A)
$$
and it is a non-zero Hecke eigen, square integrable automorphic form of weight $l$ with respect to $G(\widehat{\Z})$. 
\end{theorem}

\begin{remark}\label{standLodd} The standard $L$-function of $\Pi_{f}$ is given by
$$L(s,\Pi_{f})=L(s,\pi_f)\prod_{i=0}^{n} \zeta\Big(s+\frac n2-i\Big).
$$

\end{remark}

\begin{remark} The constant term $(F_\phi)_0$ is the Oda-Rallis-Schiffmann lift \cite{Od, RS} to $\SO(V')=\SO(2,n)$ from $f$. 
\end{remark}

\begin{remark} The global Arthur parameter of $\Pi_f$ is given in Section \ref{Lfunction}.
Such a parameter is called highly non-tempered in \cite[Section 4.2]{BMM}.
By \cite[Theorem 4.1]{BMM} together with Proposition \ref{thetaprop}, there are no cuspidal representations whose global Arthur parameters coincide with that of $\Pi_f$. 
\end{remark}

We organize this paper as follows. In Section \ref{pre}, we give explicit forms of 
the Siegel parabolic subgroup $P$ and another maximal parabolic $Q$ which gives rise to the Jacobi group $J$ in our setting. 
In Section \ref{AFSO}, we recall Pollack's definition of certain types of automorphic forms on $G(\A)$ \cite{Po}. 
In Section \ref{Theta}, we recall basic facts of the theta lifts to $G$ from holomorphic elliptic newforms of level one and their local Langlands parameters.   
In Section \ref{DWF}, we discuss the degenerate Whittaker functions at each 
place. In the construction, in order to normalize the Jacquet integrals, the computation of the 
Siegel series carried out in \cite[Section 4]{KY2} plays an important role. 
In Section \ref{weil-rep}, we briefly recall the Weil representation for $J$.  
The Fourier-Jacobi expansion along $J$ of Eisenstein series associated to $P$ fixed by $G(\widehat{\Z})$ 
is also explicitly computed in Sections \ref{generalcase} and \ref{FJEl}. 
According to the computations in the previous section, we carry out similar computations in terms of degenerate Whittaker functions in Section \ref{FJE-DWF}. Readers may skip 
Section \ref{FJEl} which  is not logically necessary for the arguments in Section \ref{FJE-DWF}, but it may be helpful for understanding the definition of local Fourier-Jacobi maps. 
Without Section \ref{FJEl}, it may be difficult for figuring out those definitions. 

Then, we give proofs for the main theorems in Section \ref{proof}. 
We also discuss the $L$-functions and Arthur-parameters of the theta lifts in Section \ref{Lfunction}. In Section \ref{residual}, we explain that the theta lift is indeed in the residual spectrum.
Finally, Appendix is given for computing some integrals which show up in Section \ref{FJEl}. 

\smallskip  

\textbf{Acknowledgments.} We would like to thank Tamotsu Ikeda, Takuya Miyazaki, Hiraku Atobe, Hiroaki Narita and Jim Arthur for helpful discussions, and Yi Shan for pointing out some inaccuracies.
We thank KIAS in Seoul and Waseda University in Tokyo for their incredible hospitality during this research. 

\smallskip

\section{Preliminaries}\label{pre}
For each quadratic space $W$ or its representation matrix $A$, the corresponding symmetric pairing is denoted by $(\ast,\ast)_W$ or $(\ast,\ast)_A$. 

For each integer $n\ge 2$, let $A$ be a positive definite symmetric matrix in 
$M_{n-2}(\Z)$ as defined in Section \ref{intro}. Put  
$$J_{1,n-1}=\left(\begin{array}{ccc}
 0 & 0 & 1 \\
 0 & -A  & 0  \\ 
1 & 0 & 0 
\end{array}\right),\ J_{2,n}=\left(\begin{array}{ccc}
 0 & 0 & 1 \\
 0 & J_{1,n-1}  & 0  \\ 
1 & 0 & 0 
\end{array}\right)
,\ 
 J_{3,n+1}=\left(\begin{array}{ccc}
 0 & 0 & 1 \\
 0 & J_{2,n}  & 0  \\ 
1 & 0 & 0 
\end{array}\right).$$

Let $V'=
\mathbb{G}^{n+2}_a$ be the quadratic space associated to 
$J_{2,n}$. 
Let 
\begin{equation}\label{vv}
v_1:=(1,\overbrace{0,\ldots,0}^{n},1),\ v_2:=(0,1,\overbrace{0,\ldots,0}^{n-2},1,0) \in V'(\Q)
\end{equation}
so that $(v_i,v_j)_{V'}=\frac{1}{2}v_iJ_{2,n}{}^t v_j=\delta_{ij}$ for $1\le i,j\le2$. 
Let $V=\mathbb{G}_a e\oplus V'\oplus \mathbb{G}_a f=\mathbb{G}^{n+4}_a$  be the quadratic space associated to 
$J_{3,n+1}$ where 
$$e:=(1,\overbrace{0,\ldots,0}^{n+3}),\ f:=(\overbrace{0,\ldots,0}^{n+3},1)$$
so that $(e,f)_V=\frac{1}{2}eJ_{3,n+1}{}^t f=\frac{1}{2}$. Let $q$ be the quadratic form on $V'$ 
defined by $q(x)=\frac{1}{2}xJ_{2,n} {}^t x$ for $x\in V'$. 
Denote by $$(x,y)=q(x+y)-q(x)-q(y)$$ for $x,y\in V'$ the associated bilinear form. The quadratic form $\tilde q$ on $V$ for $J_{3,n+1}$ 
satisfies $\tilde q(\alpha e+v'+\beta f)=\alpha\beta+q(v')$.

Let
$$G=\SO(V)=\SO(3,n+1)=\{g\in \SL_{n+4}\ |\ {}^t g J_{3,n+1}g=J_{3,n+1}\}$$ and we also consider 
$$\SO(V')=\SO(2,n)=
\{g\in \SL_{n+2}\ |\ {}^t g J_{2,n}g=J_{2,n}\}.
$$

The natural embedding $V'\hookrightarrow V,\ v'\mapsto (0,v',0)$ yields an embedding 
$$\SO(V')\hookrightarrow \SO(V),\ g'\mapsto 
\left(\begin{array}{ccc}
 1 & 0& 0 \\
 0 & g'  & 0  \\ 
0 & 0 & 1 
\end{array}\right).$$ 
We also define the Weyl element $w_0\in G(\Q)$ such that $w_0(e)=f,\ w_0(f)=e$, and 
$w_0|_{V'}={\rm id}_{V'}$. 

Throughout this paper, we assume $G=\SO(V)$ splits at any finite place of $\Q$ as in Section \ref{intro}.

\begin{remark}\label{diffePo}
Our definition of $q$ is slightly different from the one in \cite{Po}. 
Then, our $\SO(V')$ is isomorphic over $\R$ to Pollack's one but not over $\Q$ 
since any normalized orthogonal transformation matrices have to involve $\sqrt{2}$. 
\end{remark}

\subsection{Siegel parabolic subgroup $P=MN$}
We define 
$$M=\{\diag(t,m,t^{-1})\in G\ |\ t\in \GL_1,\ m\in \SO(V')\}\simeq \GL_1\times \SO(V')$$ with 
$\nu:M\lra \GL_1,\ \diag(t,m,t^{-1})\mapsto t$, 
and 
$$N=\left\{n(\textbf{x}):=
\left(\begin{array}{ccc}
 1 & -\textbf{x}J_{2,n} & -\frac{1}{2}\textbf{x}J_{2,n}{}^t\textbf{x} \\
 0 & 1_{n+2}  & {}^t\bf{x}  \\ 
0 & 0 & 1 
\end{array}\right) \Bigg|\ \textbf{x}\in \mathbb{G}^{n+2}_a
\right\}\stackrel{\sim}{\lra}V',\ n(x)\mapsto x.$$
Then, $\diag(t,m,t^{-1})n\in P$ acts on $x\in V'$ by $t(x\cdot {}^t m)$ where 
 the dot ``$\cdot$'' stands for the usual matrix multiplication. 

Once a suitable Haar measure is chosen, 
the modulus character of $P$ is given by $\delta_P(\diag(t,m,t^{-1}))=|t|^{n+2}$.

\subsection{Another maximal parabolic subgroup and the Jacobi group}\label{another}
Let $V''$ be the quadratic space associated to $J_{1,n-1}$. 
Let $$Q=\left\{
\left(\begin{array}{ccc}
 A_1 & \ast &  \ast \\
 0 & m & \ast  \\ 
0 & 0 & A_2 
\end{array}\right)\in G \Bigg|\ A_1,A_2\in \GL_2,\ m\in {\rm O}(V'')
\right\}$$
be another maximal parabolic subgroup. Explicitly, if we write $Q=LU$, then 
$$L=\left\{
\left(\begin{array}{ccc}
 \det(\gamma)^{-1}\eta \gamma\eta & 0 &  0 \\
 0 & m & 0  \\ 
0 & 0 & \gamma 
\end{array}\right)\ \Bigg|\ \gamma\in \GL_2,\ m\in \SO(V'')
\right\}\simeq \GL_2\times \SO(V'')
$$
where $\eta=\left(\begin{array}{cc}
 -1 & 0   \\ 
0 & 1 
\end{array}\right)$ 
and the unipotent radical $U$ is consisting of 
all $v(x,y,z):=v(x,y,0)v(0,0,z)\in G$ where 
$$v(x,y,0)=
\left(\begin{array}{ccccc}
1 & 0 & -yJ_{1,n-1}  & -\frac{1}{2}xJ_{1,n-1}{}^t y & -\frac{1}{2}yJ_{1,n-1}{}^t y  \\ 
0 & 1 & -xJ_{1,n-1}  & -\frac{1}{2}xJ_{1,n-1}{}^t x &-\frac{1}{2} yJ_{1,n-1}{}^t x  \\ 
0 & 0 & 1_{n}  & {}^t x & {}^t y  \\ 
0 & 0 & 0  &1 & 0  \\ 
0 & 0 &0 &0 & 1  \
\end{array}\right),\ 
v(0,0,z)=
\left(\begin{array}{ccccc}
1 & 0 & 0  & -z & 0  \\ 
0 & 1 & 0  & 0 & z  \\ 
0 & 0 & 1_{n}  & 0 & 0  \\ 
0 & 0 & 0  &1 & 0  \\ 
0 & 0 &0 &0 & 1  \
\end{array}\right)$$
for 
$x,y\in \mathbb{G}^{n}_a, z\in \mathbb{G}_a$. It is easy to see that $v(0,y,z)=n(\textbf{x})\in N$ with $\textbf{x}=(z,y,0)$.  
Then, 
$$U=\{v(x,y,z)\in G\ |\ x,y\in \mathbb{G}^{n}_a, z\in \mathbb{G}_a\}$$ is a Heisenberg group with the center $Z=\{v(0,0,z)\ |\ z\in \mathbb{G}_a\}$ and we have 
$$v(x,y,z)v(x',y',z')=v(x+x',y+y',z+z'+\frac{1}{2}\langle (x,y),(x',y')\rangle)$$
where $\langle (x,y),(x',y')\rangle:=xJ_{1,n-1}{}^ty'-x'J_{1,n-1}{}^ty$ 
defines a symplectic form on $XY$ with $X=\{v(x,0,0)\ |\ x\in \mathbb{G}^n_a\},\ 
Y=\{v(0,y,0)\ |\ y\in \mathbb{G}^n_a\}$. We identify $X$ with $V''$. 
Put 
\begin{equation}\label{sym}
\sigma(x,y):=\langle (x,0),(0,y)\rangle=x J_{1,n-1}{}^t y
\end{equation}
which shows up in the action of the Weil representation in our setting.

We define 
\begin{equation}\label{mapl}
H:=\Bigg\{\ell(\gamma)\in L\ \Bigg|\ \gamma=
\left(\begin{array}{cc}
 a & b   \\ 
c & d  
\end{array}\right)\in \SL_2\Bigg\}\simeq \SL_2,\ \ell(\gamma)\longleftrightarrow \gamma
\end{equation}
where 
$$\ell(\gamma)=\left(\begin{array}{ccc}
\det(\gamma)^{-1}\eta \gamma\eta & 0  & 0    \\ 
0 & 1_{n}  & 0   \\ 
0 & 0 &  \gamma
\end{array}\right),\ \eta=\left(\begin{array}{cc}
 -1 & 0   \\ 
0 & 1 
\end{array}\right).$$ 
Note that 
\begin{equation}\label{elldiag}
\ell(\diag(a,a^{-1}))=\diag(a^{-1},1_{n+2},a)\in M. 
\end{equation}

It is easy to see that 
$$\ell^{-1}(\gamma)v(x,y,z)\ell(\gamma)=v(ax+cy,bx+dy,\det(\gamma)z),\ \gamma=
\left(\begin{array}{cc}
 a & b   \\ 
c & d  
\end{array}\right)\in \GL_2,\ x,y\in \mathbb{G}^n_a,\ 
z\in \mathbb{G}_a$$
as in \cite[(4.5),p.231]{KY}. Thus, we have the Jacobi group in our setting:
$$J:=U\rtimes H.
$$

\section{Automorphic forms on $G=\SO(3,n+1)$}\label{AFSO}

Recall $G=\SO(V)$. 
The maximal compact subgroup $K$ of $G(\Bbb R)$ is $S({\rm O}(3)\times {\rm O}(n+1))$. 
Let $K_0=\SO(3)\times \SO(n+1)$.  
The projection onto the first factor induces a surjective homomorphism $S({\rm O}(3)\times {\rm O}(n+1))\rightarrow {\rm O}(3)=({\rm SU}(2)/\mu_2)\rtimes \{\pm 1\}$. For each $\lambda\in \C$, we consider the normalized induced representation 
$$I(\lambda):={\rm Ind}^{G(\R)}_{P(\R)} \, |\nu|^{\lambda-\frac{n}{2}-1}.$$
By the Peter-Weyl theorem, we have 
$${\rm Ind}^{K_0}_{K_0\cap P(\R)}\mathbf{1}\simeq 
C^\infty((S^2\times S^{n})/\{\pm 1\})\simeq \bigoplus_{a,b\ge 0\atop a+b:\, {\rm even}}
\mathcal{H}^{a,b},\ 
\mathcal{H}^{a,b}:=\mathcal{H}^a(\C^3)\boxtimes \mathcal{H}^b(\C^{n+1})$$
where $S^k$ stands for the $k$-dimensional sphere and $\mathcal{H}^i(\C^j)$ 
stands for the harmonic polynomials on $\C^j$ of degree $i$. Using this, we have 
\begin{equation}\label{K-type}
I(\lambda)|_K= \bigoplus_{a,b\ge 0\atop a+b:\, {\rm even}}{\rm Ind}^K_{K_0}
\mathcal{H}^{a,b}.
\end{equation}
If the twist of $\mathcal{H}^{a,b}$ by $K/K_0$ is isomorphic to $\mathcal{H}^{a,b}$, then 
$${\rm Ind}^K_{K_0}
\mathcal{H}^{a,b}=\mathcal{H}^{a,b,+}\oplus \mathcal{H}^{a,b,-}$$
where $\mathcal{H}^{a,b,\ve}$ is an extension of $\mathcal{H}^{a,b}$ to $K$ with the 
action of a character $\ve:K/K_0\lra \{\pm 1\}$. Otherwise, ${\rm Ind}^K_{K_0}
\mathcal{H}^{a,b}$ is irreducible. 
 In particular, if $b=0$ (hence $a$ is even, say $a=2l$), then 
$\mathcal{H}^{a,b}$ admits a natural extension to a representation 
$\mathcal{H}^{a,b,+}$ of $K$; by abuse of notation, we denote this extension again by 
$\mathcal{H}^{a,b}$. It appears in $I(\lambda)|_K$ with multiplicity one.

\begin{defin}\label{MFG} Let $l\ge 0$ be an integer. Let $\tau_l:K\lra \GL(\Bbb V_l),\ \Bbb V_l={\rm Sym}^{2l}(\Bbb C^2)$ be the $(2l+1)$-dimensional algebraic representation of $K$ that factors through ${\rm O}(3)$ and we regard it with 
 $\mathcal{H}^{2l}(\C^3)\boxtimes \mathbf{1}$ as a $K$-type of $I(l+1)$. 
We fix a basis $\{e_v:=[x]^{l+v}[y]^{l-v}\}_{v=-l}^l$ of $\Bbb V_l$ where  $[x^j]=\frac {x^j}{j!}$ and  $[y^j]=\frac {y^j}{j!}$ for 
any integer $j\ge 0$. 
Modular forms on $G$ of weight $l$ are $\Bbb V_l$-valued automorphic functions $\phi$ on $G(\Bbb A)$ that satisfy:

(1) $\phi(g k)=\tau_l(k^{-1})\phi(g)$ for all $g\in G(\Bbb A)$ and $k\in K$. 

(2) $\phi$ is annihilated by a special differential operator $\mathcal D_l$. 

(3) As a $(\frak g, K)$-module, $\phi$ generates an irreducible constituent of $I(l+1)$. 

Cusp forms on $G$ of weight $l$ are similarly defined as in \cite{BJ}.  
\end{defin}

\begin{prop}\label{inf-type}Assume $l>\frac{n}{2}$ is an even integer. 
Let $F$ be a non-zero modular form on $G$ of weight $l$ and $\Pi_\infty$ be 
the corresponding irreducible admissible representation of $G(\R)$. 
Then, as a $(\frak g,K)$-module, $\Pi_\infty$ is isomorphic to a unique irreducible component of $\Pi^{3,n+1}_{l,0}|_{G(\R)}$ 
containing $\mathcal{H}^{2l}(\C^3)\boxtimes \mathbf{1}$ as a minimal $K$-type.  
(see \cite[Proposition 2.1]{MS} for $\Pi^{3,n+1}_{l,0}$).
\end{prop}
\begin{proof}By definition, $\Pi_\infty$ is a constituent of $I(l+1)$ with a 
$K$-type $\mathcal{H}^{2l}(\C^3)\boxtimes \mathbf{1}$.  
Since $l>\frac{n}{2}$ and $l$ is even, by \cite[Proposition 2.1]{MS},  there is a unique irreducible submodule $\Pi^{3,n+1}_{l,0}$ of $I(l+1)$ with the minimal $K$-type  
 $\mathcal{H}^{2l}(\C^3)\boxtimes \mathbf{1}$. As observed, such a $K$-type is multiplicity free in $I(l+1)$. Thus, 
$\Pi_\infty$ and a unique irreducible component of $\Pi^{3,n+1}_{l,0}|_{G(\R)}$ share the same $K$-type. 
The claim now follows from irreducibility of $\Pi_\infty$.
\end{proof}

\begin{theorem}\cite{Po} Suppose $\phi$ is a modular form of weight $l\geq 1$ on $G$. Then 
$$\phi(g)=\phi_0(g)+\sum_{\eta\in V'(\Bbb Q)\atop q(\eta)\geq 0, \eta\neq 0} a_\phi(\eta)(g_{\f})\mathcal W_{2\pi \eta}(g_\infty),
$$
for $g=g_{\f}g_\infty$ in $G(\Bbb A_{\f})\times G(\Bbb R)$, where $a_\phi(\eta): G(\Bbb A_{\f})\longrightarrow \Bbb C$ is a locally constant function. Moreover,
$$\phi_0(m)=\Phi(m)x^{2l}+\beta(m_f)x^ly^l+\Phi'(m)y^{2l},
$$
where $\Phi$ is an automorphic function associated to a holomorphic modular form of weight $l$ on $M$, $\beta$ is a locally constant function on $M(\Bbb A_{\f})$, and $\Phi'$ is a certain $(K\cap M)$-right translate of $\Phi$.
\end{theorem}

Here $\mathcal W_{\eta}(g): G(\Bbb R)\longrightarrow \Bbb V_l$ is a generalized Whittaker function of type $\eta$ satisfying 
\begin{enumerate}
\item $\mathcal W_\eta(n(x)g)=e^{i(\eta,x)} \mathcal W_\eta(g)$;

\item $\mathcal W_\eta(gk)=k^{-1} \mathcal W_\eta(g)$;

\item $D_l \mathcal W_\eta(g)=0$;

\item Suppose $\eta\in V'(\Bbb R)$, $\eta\ne 0$, and $(\eta,\eta)\geq 0$. For  $t\in GL_1(\Bbb R), m\in \SO(V')(\Bbb R)$, set
\begin{equation}\label{ourueta}
u_\eta(t,m)=it(\eta,m(iv_1-v_2)).
\end{equation}
We should remark that Pollack used $u_\eta(t,m)=it\sqrt{2}(\eta,m(iv_1-v_2))$ in \cite{Po} 
because of the shape of $q$ (see also Remark \ref{diffePo}). 
 Then
\begin{equation}\label{PollackSpherical}
\mathcal W_\eta(t,m)=t^l |t| \sum_{-l\leq v\leq l} \left(\frac {|u_\eta(t,m)|}{u_\eta(t,m)}\right)^v K_v(|u_\eta(t,m)|) [x^{l+v}][y^{l-v}],
\end{equation}
where $K_v$ is the $K$-Bessel function defined by
$$K_v(y)=\frac 12\int_0^\infty e^{-y(t+t^{-1})/2} t^v\, \frac {dt}t.
$$
\end{enumerate}

\begin{defin}\label{PollackSection} We define the Pollack's section $\phi^{{\rm Po}} \in I(l+1)$ by 
$$\phi^{{\rm Po}}(g)=|t|^{l+1} \Big(\sum_{v=-l}^l \tau_l(k^{-1})e_v\Big),\ g=\diag(t,m,t^{-1})nk\in G(\R).$$
By the previous proposition, it belongs to a unique irreducible component of $\Pi^{3,n+1}_{l,0}|_{G(\R)}$ 
containing $\mathcal{H}^{2l}(\C^3)\boxtimes \mathbf{1}$ as a minimal $K$-type. 
For any $\eta\in V'(\R)$ with $q(\eta)>0$, define the Jacquet integral 
$$g\mapsto \int_{N(\R)}\phi^{{\rm Po}}(w n(x)g)\psi((\eta,x))dx$$
where we identify $V'(\R)$ with $N(\R)$. 
By \cite[Theorem 3.2.4 and Theorem 4.5.9]{Po}, it is a non-zero constant multiple of 
$\mathcal{W}_{2\pi \eta}(g)$. The constant is in dependent of $\eta$ because $M(\R)$ acts transitively on 
the set $\{\eta\in V(\R)\ |\ q(\eta)>0\}$.   
\end{defin}


The following fact is not used in this paper, but may be useful for the study of modular forms on $G$.

\begin{prop}\label{cuspforms} Let $\phi$ be a modular form of weight $l\ge 1$ and be of level one, hence it is fixed by $G(\widehat{\Z})$. 
If it has the Fourier expansion of the form
$$\phi(g)=\sum_{\eta\in V'(\Bbb Q)\atop q(\eta)>0} a_\phi(\eta)(g_f)\mathcal W_{2\pi \eta}(g_\infty), 
$$
then $\phi$ is a cusp form. 

\end{prop}
\begin{proof} By assumption, $G$ splits everywhere at finite places. Thus, 
we have $G(\Q)=R(\Q)G(\Z)$ for any maximal standard $\Q$-parabolic subgroup  $R$ of $G$. 
Therefore, we may check the cuspidality only for three maximal standard 
$\Q$-parabolic subgroups. 

Since $G(\A)=G(\Q)(G(\widehat{\Z})\times G(\R))$, we may write 
 $$\phi(g)=\sum_{\eta\in V'(\Bbb Q)\atop q(\eta)>0} 
 a_\phi(\eta)(1_f)\mathcal W_{2\pi \eta}(g_\infty),\ g=\gamma k_f g_\infty. 
$$
For each  maximal standard $\Q$-parabolic subgroup $R$ with the unipotent radical $U_R$, by using the decomposition 
$G(\A)=R(\Q)(G(\widehat{\Z})\times G(\R))$, we have 
\begin{eqnarray*}
\int_{U_R(\Q)\bs U_R(\A)}\phi(ug)du &=&\int_{U_R(\Z)\bs U_R(\R)}\phi(u_\infty g_\infty)du_\infty \\
&=&\sum_{\eta\in V'(\Bbb Q)\atop q(\eta)>0} 
 a_\phi(\eta)(1_f)
 \int_{U_R(\Z)\bs U_R(\R)}
 \mathcal W_{2\pi \eta}(u_\infty g_\infty)du_\infty.
\end{eqnarray*}
Thus, we may check the vanishing of $\ds\int_{U_R(\Z)\bs U_R(\R)}
 \mathcal W_{2\pi \eta}(u_\infty g_\infty)du_\infty$ for each $R$. 
 
Since $q(\eta)\neq 0$, the constant term along $N$ (the unipotent radical of $R=P$) is vanishing. 

When $R=Q=LU$, let $U_1=\{v(0,y,z)\in U\}\supset U_2=\{v(0,y,z)\in U\ |\ y=(0,\ldots,0,y_n)\}$. Note that $U_1$ is a normal subgroup so that $U/U_1\simeq X$ and $U_2$ is a normal subgroup of $U_1$ so that $U_2/U_1\simeq Y':=\{v(0,y,0)\ |\ y=(y_1,\ldots,y_{n-1},0)\}$.   
If we write $\eta=(\eta_1,\eta_2,\xi,\eta_{n+1},\eta_{n+2})\in V'(\Q)$, then
$q(\eta)=\eta_1\eta_{n+2}+\eta_2\eta_{n+1}-\xi A {}^t\xi$. Thus, $\eta_1\eta_{n+2}+\eta_2\eta_{n+1}\neq 0$ 
since $q(\eta)>0$. Thus, we have $(\eta_1,\eta_2)\neq (0,0)$. 
Then $\ds\int_{U_R(\Z)\bs U_R(\R)}
 \mathcal W_{2\pi \eta}(u_\infty g_\infty)du_\infty=$
 $$
 \int_{X(\Z)\bs X(\R)}\Bigg(
 \int_{U_1(\Z)\bs U_1(\R)}\psi((\eta,u_{1,\infty}))du_{1,\infty} \Bigg)
 ) \mathcal W_{2\pi \eta}(v(x_\infty,0,0) g_\infty) dx_\infty.$$
The inner integral becomes  
$$\int_{Y'(\Z)\bs Y'(\R)}\psi(-y'_\infty A {}^t \xi)d y'_\infty\times 
\int_{U_2(\Z)\bs U_2(\R)}\psi(\eta_1 z_\infty+\eta_2 y_{n,\infty})dz_\infty dy_{n,\infty}=0$$
since the second integral vanishes by $(\eta_1,\eta_2)\neq (0,0)$.

Now let $R=M_R U_R$ with
$M_R\simeq \GL_3\times \SO(n-2)$.
Let $Z'$ be the center of $U_R$. It is easy to see that $Z'$ contains $U_2$.
Thus, the vanishing follows by a similar argument to that in the previous case.
\end{proof}

By definition, vectors $\eta\in V'(\Bbb Q)$ with $q(\eta)>0$ are anisotropic vectors and they are called rank 2 elements in \cite{Po}.

\section{Theta lifts to $\OO(V)$ and its restriction to $G$}\label{Theta}

Throughout this section we assume $n\ge 2$ so that the case $n=2$ is allowed. 
Let $l>n+1$ be an even integer and let $f=\ds\sum_{n=1}^\infty 
n^{\frac{k-1}{2}}a_f(n) q^n $ be a normalized Hecke eigen cusp form of weight 
$k=\left\{\begin{array}{ll}
l-\frac{n}{2}+1 & \text{for $n$ even,}\\ 
2l-n+1  & \text{for $n$ odd}
\end{array}\right.
$
with respect to $\SL_2(\Z)$. Let $a_f(p)=\alpha_p+\alpha^{-1}_p,\ \alpha_p\in U(1)=\C^{1}$. 

Let $\psi:\Q\bs\A\lra \C^\times$ be the  standard additive character.  
Let $\pi_f$ be the irreducible unitary cuspidal representation of  $\SL_2(\A)$ (when $n$ even) 
or of the metaplectic double covering ${\rm Mp}_2(\A)$ of $\SL_2(\A)$ (when $n$ odd) attached to $f$ 
as explained in Section \ref{intro}.  
Let $V$ be the quadratic space associated to $J_{3,n+1}$. Note that $\disc(V)=-1$. 
Let $\chi_V=\otimes'_{p}\chi_{V,p}:\A^\times\lra \{\pm 1\}$ be the quadratic character associated to 
the quadratic extension $\Q(\sqrt{\det(V)})/\Q$. 
  
In this section, we discuss the theta lift $\theta(\pi_f)=\theta^\psi(\pi_f)$ from 
$\pi_f$ to $H(\A)={\rm O}(V)(\A)={\rm O}(3,n+1)(\A)$ and its restriction to $G(\A)$. We refer to \cite[Section 10]{Yamana} 
(resp. \cite{AG}) for the basic properties of the global (resp. local) theta correspondence in a  
general setting. 
Let  $\Pi_f$ be an irreducible component of $\theta(\pi_f)|_{G(\A)}$ is irreducible. 
Then, we will prove the following result:
\begin{prop}\label{thetaprop}Keep the notation as above. It holds that 
\begin{enumerate}
\item $\Pi_f$ is non-zero,  non-cuspidal, but square-integrable;
\item for each prime $p$, if $n$ is odd, then 
$\Pi_{f,p}\simeq {\rm Ind}^{G(\Q_p)}_{P(\Q_p)}\mu_p\circ \nu$, and  if $n$ is even, then 
$\Pi_{f,p}$  is a unique unramified summand of ${\rm Ind}^{G(\Q_p)}_{P(\Q_p)}\mu^2_p\circ \nu$;
\item twisting $\Pi_f$ by an outer automorphism if necessary, $\Pi_{f,\infty}$ can be a unique irreducible component of $\Pi^{3,n+1}_{l,0}|_{G(\R)}$ containing 
$\mathcal{H}^{2l}(\C^3)\boxtimes \mathbf{1}$ as a minimal $K_\infty$-type;
\item $\Pi_{f,\infty}$ is non-tempered, cohomological. 
\end{enumerate}
\end{prop}

We redefine $\Pi_f$ as a unique irreducible component of $\theta(\pi_f)|_{G(\A)}$ such that 
$\Pi_{f,\infty}$ satisfies the third condition of Proposition \ref{thetaprop}. 


\subsection{Non-cuspidality} 
We write $V=\mathbb{H}^3\oplus U$ where $\mathbb{H}$ is the hyperbolic plane and 
$U$ is the quadratic space corresponding to $-A$. 
Let 
$$\overline{\SL}_2(\A)=\left\{\begin{array}{ll}
\SL_2(\A) & \text{for $n$ even,}\\ 
{\rm Mp}_2(\A)  & \text{for $n$ odd.}
\end{array}\right.
$$
When $\overline{\SL}_2(\A)={\rm Mp}_2(\A)$ (hence, when $n$ odd),  
we consider only genuine representations so that they do not factor through $\SL_2(\A)$. 
\begin{lemma}\label{nonvan}
Assume $n\ge 2$. 
Let $\pi$ be an irreducible cuspidal representation of $\overline{\SL}_2(\A)$. Then, 
the theta lift $\theta(\pi)$ of $\pi$ to $H(\A)={\rm O}(3,n+1)$ is nonzero and non-cuspidal, but 
square-integrable. 
\end{lemma}   
\begin{proof}Applying \cite[p.712, Proposition 10.1(3)]{Yamana} for 
$j=3,\ n=2$, and $U[j]=V$ in the notation therein, we first see that $\theta(\pi)$ is nonzero. 
By Rallis tower property (see the proof of \cite[p.712, Proposition 10.1]{Yamana}), 
the cuspidality of $\theta(\pi)$ is equivalent to the vanishing of the theta lift to ${\rm O}(2,n)={\rm O}(\mathbb{H}^2\oplus U)$. However,  
applying \cite[p.712, Proposition 10.1(3)]{Yamana} again for 
$j=2,\ n=2$, and $U[j]=\mathbb{H}^2\oplus U$ in the notation therein, we see that $\theta(\pi)$ is nonzero. 
Thus, $\theta(\pi)$ is non-cuspidal. 

Finally, applying \cite[p.712, Proposition 10.1(4)]{Yamana} for 
$m=n-2,\ j=3$ with our $n\ge 2$, and $\rho_2=3,\ j_0\ge 2$ in the notation therein 
(notice that $m+j+j_0\ge m+j+2=n+j\ge 5>\rho_2=3$), we see that 
$\theta(\pi)$ is square integrable.     
\end{proof}

\begin{remark} Lemma \ref{nonvan} is just a direct consequence of the well-known facts 
(cf. \cite[p.712, Proposition 10.1]{Yamana}) and it is also true for more general setting as explained  therein. 
\end{remark}

The first claim of Proposition \ref{thetaprop} follows from Lemma \ref{nonvan} since $\OO(V)=\SO(V)\rtimes \{\pm1\}$ as an algebraic group over $\Q$. 

\subsection{The local L-parameter at non-archimedean place when $n$ even}  
By assumption, $G$ splits at any finite places. Thus, $G(\Q_p)\simeq 
\SO(\frac{n}{2}+2,\frac{n}{2}+2)$. 
Let $\pi=\pi_\infty\otimes\otimes'_{p}\pi_p$ be a unique generic irreducible unitary cuspidal representation of $\overline{\SL}_2(\A)=\SL_2(\A)$ 
attached to $f$. For each finite place $p$, we denote by $\phi({\rm Sym}^2(\pi_p)):L_{\Q_p}:=W_{\Q_p}
\times \SL_2(\C)\lra  \SO(3)(\C)$ the $L$-parameter of ${\rm Sym}^2(\pi_p)$ where $W_{\Q_p}$ 
stands for the Weil group of $\Q_p$. Explicitly, 
it is given by 
$$\phi({\rm Sym}^2(\pi_p))({\rm Frob}_p,g)=\diag(\alpha^2_p,1,\alpha^{-2}_p),\ g\in \SL_2(\C)
$$ 
up to 
conjugation by the action of the Weyl group of $\SO(3)(\C)$. 

For each prime $p$, $\pi_p$ is an unramified principal series and 
$m^{{\rm down}}(\pi_p)=4$ with $\kappa=1,\ l(\pi_p)=-1$ by \cite[Theorem 4.1]{AG} in 
the notation therein. Applying \cite[Theorem 4.3(2)(4)]{AG} with $m_1=4$ and $\epsilon_0=1$, the local L-parameter of 
$\theta(\pi_p)\otimes \chi^{-1}_{V}$ is given by 
\begin{equation}\label{localpara}
(\phi({\rm Sym}^2(\pi_p))\oplus \mathbf{1})\oplus 
\bigoplus_{i=1}^{\frac{n}{2}}(|\cdot|^{\frac{n+2}{2}-i}_p\oplus 
|\cdot|^{-\frac{n+2}{2}+i}_p)
\end{equation}
where $\chi_V:\Q^\times_p\lra \C^\times$ is the quadratic character associated to 
$\Q_p(\sqrt{\disc(V)})/\Q_p$ (note that $\disc(V)=-1$ by definition). 
See also \cite[p. 4789, (4.2)]{CZ}. 

Since we have considered the restriction to $G(\A)$, the effect of $\chi_V$ disappears. 
Thus, the local L-parameter of $\Pi_{f,p}$ is given by (\ref{localpara}).  

By  \cite[Proposition 5.6]{AG} (applying it with $n=2,m=4,\epsilon=1$, 
and $r=n$), we see that $\Pi_{f,p}$ is the unique irreducible quotient of 
the standard module
$${\rm Ind}^{G(\Q_p)}_{R(\Q_p)}(|\cdot|^{\frac{n}{2}}\otimes 
|\cdot|^{\frac{n}{2}-1}\otimes \cdots\otimes |\cdot|^{1}\rtimes \tau_p)$$
where  $R$ is a parabolic subgroup whose Levi subgroup is isomorphic to 
$GL^{\frac{n}{2}}_1\rtimes {\rm SO}(2,2)$ and $\tau_p$ is some irreducible (unramified) component of  $\theta_{V_4,W_2}(\pi_p)|_{\SO(2,2)}$. 
Here $\theta_{V_4,W_2}(\pi_p)$ is the local theta 
correspondence of $\pi_p$ to ${\rm O}(2,2)={\rm O}(V_4)$ whose L-parameter is given by 
$\phi({\rm Sym}^2(\pi_p))\oplus \mathbf{1}$ 
(see \cite[Theorem 4.3(3)]{AG}). 
Since $\tau_p$ is (isomorphic to) a quotient of ${\rm Ind}^{\SO(2,2)}_{B_{\SO(2,2)}}\mu^2_p\otimes \mathbf{1}$ where $B_{\SO(2,2)}$ is the 
standard Borel subgroup of $\SO(2,2)$, $\Pi_p$ is (isomorphic to) the quotient of 
$$X_p:={\rm Ind}^{G(\Q_p)}_{B(\Q_p)}(|\cdot|^{\frac{n}{2}}\otimes 
|\cdot|^{\frac{n}{2}-1}\otimes \cdots\otimes |\cdot|^{1}\otimes \mu^2_p\otimes \mathbf{1})$$
where $B(\Q_p)$ is the standard Borel  subgroup of $G(\Q_p)$. 

Let $H=\SO(\frac{n}{2}+1,\frac{n}{2}+1)\subset G$ with the standard Borel subgroup $B_H$. 
Since the trivial representation $\mathbf{1}_H$ of $H(\Q_p)$ is the Langlands quotient of 
${\rm Ind}^{H(\Q_p)}_{B_H(\Q_p)}(|\cdot|^{\frac{n}{2}}\otimes 
|\cdot|^{\frac{n}{2}-1}\otimes \cdots\otimes |\cdot|^{1}\otimes\mathbf{1})$, by induction of stages 
the representation ${\rm Ind}^{G(\Q_p)}_{P(\Q_p)}\mu^2_p\circ \nu$ is a unique quotient of 
$${\rm Ind}^{G(\Q_p)}_{P(\Q_p)}(\mu^2_p \boxtimes ({\rm Ind}^{H(\Q_p)}_{B_H(\Q_p)}(|\cdot|^{\frac{n}{2}}\otimes 
|\cdot|^{\frac{n}{2}-1}\otimes \cdots\otimes |\cdot|^{1}\otimes\mathbf{1})))=
{\rm Ind}^{G(\Q_p)}_{B(\Q_p)}(\mu^2_p\otimes |\cdot|^{\frac{n}{2}}\otimes 
|\cdot|^{\frac{n}{2}-1}\otimes \cdots\otimes |\cdot|^{1}\otimes \mathbf{1})\simeq X_p$$
where the latter isomorphism is given by the Weyl action. 
Thus, both $\Pi_p$ and ${\rm Ind}^{G(\Q_p)}_{P(\Q_p)} \mu_p^2\circ \nu$ are unramified, irreducible and arise 
as quotients of the same standard module. Therefore, they are isomorphic by the uniqueness of the Langlands quotient (cf. \cite{Konno}).

\subsection{The local L-parameter at non-archimedean place when $n$ odd}  
By assumption, $G$ splits at any finite places. Thus, $G(\Q_p)\simeq 
\SO(\frac{n+3}{2},\frac{n+3}{2}+1)$. 
Let $\pi=\pi_\infty\otimes\otimes'_{p}\pi_p$ be the irreducible unitary cuspidal representation of 
$\overline{\SL}_2(\A)={\rm Mp}_2(\A)$ attached to the Shimura-Waldspurger lift of $f$ via $PGL_2\simeq SO(1,2)$. 
For each finite place $p$, we denote by $\phi(\pi_p):L_{\Q_p}:=W_{\Q_p}
\times \SL_2(\C)\lra  \SL_2(\C)$ the $L$-parameter of $\phi_p$. Explicitly, 
it is given by $$\phi(\pi_p)({\rm Frob}_p,g)=\diag(\alpha_p,\alpha^{-1}_p),\ g\in \SL_2(\C)$$ up to 
conjugation by the action of the Weyl group of $\SL_2(\C)$. 

Let $V_3$ be the quadratic space attached to $\begin{pmatrix} 0&0&1\\0&1&0\\1&0&0\end{pmatrix}$ so that $\disc(V_3)=-1$. Let $\chi_{V_3}:\A^\times\lra \{\pm 1\}$ be the quadratic character 
associated to $V_3$ so that $\chi_{V_3}=\chi_V$. 
For each prime $p$, we see that $V\simeq V_3\oplus \mathbb{H}^{\frac{n+1}{2}}$ over $\Z_p$. 
For each prime $p$, $\pi_p$ is an unramified principal series and 
$m^{{\rm down}}(\pi_p)=3$ with $\kappa=2,\ l(\pi_p)=0$ by \cite[Theorem 4.1]{AG} in 
the notation therein. Applying \cite[Theorem 4.3(3)(4)]{AG} with $m_1=3$ and $\epsilon_0=1$, the L-parameter of 
$\theta(\pi)_p$ is given by 
$$\phi(\pi_p)\otimes \chi^{-1}_{V_3}\oplus 
\bigoplus_{i=1}^{\frac{n+1}{2}}(|\cdot|^{\frac{n+2}{2}-i}_p\oplus 
|\cdot|^{-\frac{n+2}{2}+i}_p)$$
which takes the values in ${\rm Sp}_{n+3}(\C)$. 
Since we have considered the restriction to $G(\A)$, the effect of $\chi_{V_3}$ disappears. 
Thus, the local L-parameter of $\Pi_{f,p}$ is given by 
$$\phi(\pi_p)\oplus 
\bigoplus_{i=1}^{\frac{n+1}{2}}(|\cdot|^{\frac{n+2}{2}-i}_p\oplus 
|\cdot|^{-\frac{n+2}{2}+i}_p).$$

By  \cite[Proposition 5.6]{AG} (applying it with $n=2,m=3,\epsilon=1$, 
and $r=n$), $\theta(\pi_p)$ is the unique irreducible quotient of 
the standard module
$${\rm Ind}^{G(\Q_p)}_{R(\Q_p)}(|\cdot|^{\frac{n}{2}}\otimes 
|\cdot|^{\frac{n}{2}-1}\otimes\cdots \otimes |\cdot|^{\frac{1}{2}}\rtimes \tau_p)$$
where  $R$ is a parabolic subgroup whose Levi subgroup is isomorphic to 
$GL^{\frac{n+1}{2}}_1\rtimes {\rm SO}(V_3)$ and $\tau_p$ is some irreducible  (unramified)  component of $\theta_{V_3,W_2}(\pi_p)|_{\SO(V_3)}$. Here, 
 $\theta_{V_3,W_2}(\pi_p)$ is the local theta 
correspondence of $\pi_p$ to ${\rm O}(V_3)$ whose L-parameter is given by $\phi(\pi_p)\otimes\chi^{-1}_{V_3,p}$ 
(in particular, the local L-parameter of $\tau_p$ is given by $\phi(\pi_p)$).  
Since $\tau_p$ is a quotient of ${\rm Ind}^{\SO(1,2)}_{B_{\SO(1,2)}}\mu_p$ where $B_{\SO(1,2)}$ is the 
standard Borel subgroup of $\SO(1,2)$, 
As in the previous case, it is easy to see that ${\rm Ind}^{G(\Q_p)}_{P(\Q_p)}\mu_p\circ \nu$ is also a quotient of the above standard module. 
Since both $\Pi_p$ and ${\rm Ind}^{G(\Q_p)}_{P(\Q_p)} \mu_p\circ \nu$ are unramified and arise 
as quotients of the same standard module. 

\subsection{The archimedean component}\label{check-infinite}
It follows from \cite[Section 2.2]{Kobayashi} that 
as a $(\frak g,K)$-module, $\Pi^{3,n+1}_{l,0}$ is identified with the derived 
functor module $A_{\frak q}(\lambda-\rho)=A_{\frak q}(l-n-1)$ with $\lambda=l-\frac{n}{2}$ and $\rho=\frac{n}{2}+1$ in 
the notation therein. Here $\frak q=\frak l_\C+\frak u$ is a $\theta$-stable parabolic subalgebra of $\frak g={\rm Lie}(G(\R))_\C$ with the Levi subalgebra $\frak l_C$ and the normalizer of $\frak u$ in $G(\R)$ is isomorphic to $\SO(2)\times {\rm O}(1,n+1)$. Thus, $\Pi^{3,n+1}_{l,0}$ is cohomological. 
Further, one can check $[\frak l_\C,\frak l_\C]\not\subset \frak k:={\rm Lie}(K_\C)$ by using 
the description of $\frak u$ given in \cite[p.7, line 8 from Section 2.2]{Kobayashi}. 
Thus, as explained in \cite[p.58, line -9]{VZ}, $A_{\frak q}(l-n-1)$ is non-tempered. 
Further, by \cite[p.209, Theorem 5.8]{Li}, $\theta(\pi_f)_\infty\simeq A_{\frak q}(l-n-1)$. 
Therefore, $\theta(\pi_f)_\infty\simeq \Pi^{3,n+1}_{l,0}$ is non-tempered and cohomological. 

As in the proof of \cite[p.180, Corollary 1.3]{Li}, it is easy to see that ${\rm dim}_\C(\frak u\cap \frak p)=n+1$ where 
$\frak g=\frak k+\frak p$ is the Cartan decomposition. 
Applying \cite[p.64, Theorem 3.3]{VZ} with $R={\rm dim}_\C(\frak u\cap \frak p)=n+1$, we have 
$H^{n+2}(\frak g,K,A_{\frak q}(\lambda-\rho))\neq 0$. 
Notice $\theta(\pi_f)_\infty\otimes E^\vee_l$ contains the trivial representation as a $(\frak g,K)$-module. Thus, 
$\theta(\pi_f)$ contributes to $H^{n+2}(\Gamma\bs X,\mathcal{E}_l)$ where 
$\G\bs X$ is the arithmetic variety associated to $\OO(V)$ with the level $\OO(V)(\widehat{\Z})$ and 
$\mathcal{E}_l$ is the local system corresponding to the algebraic representation $E_l={\rm Sym}^l(\C^{n+4})_{{\rm harm}}$ of 
$\OO(V)(\R)$. 
Here the subscript ``harm'' means the harmonic polynomials in ${\rm Sym}^l(\C^{n+4})$ with respect to the variables of $\C^{n+4}$. 

Since any irreducible component $\Pi_\infty$ of $\Pi^{3,n+1}_{l,0}|_{G(\R)}$ are transitive under the outer automorphism of $\OO(V)$, 
obviously $\Pi_\infty$ is non-tempered and cohomological.  
The claims (3),(4) of Proposition \ref{thetaprop} follow from that with Proposition \ref{inf-type}.

\subsection{The vanishing of Fourier coefficients at the rank one indices}\label{VanishingRankOne} 
We say an index $\eta\in V'(\Q)$ is of rank one if $\eta\neq 0$ and $q(\eta)=0$. 
It follows from \cite[Theorem 5.4]{MS} that there exists a non-zero automorphic form $F\in \theta(\pi_f)\subset 
L^2(\OO(V)(\Q)\bs \OO(V)(\A))$  such that 
$$F_\eta(g):=\int_{N(\Q)\bs N(\A)}F(ng)\overline{\psi((\eta,n))}dn\equiv 0$$
for any rank one index $\eta\in V'(\Q)$.
Since $\theta(\pi_f)$ is irreducible, the same is true for any automorphic form belonging to $\theta(\pi_f)$ and for its restriction to $G(\A)$.

\section{Degenerate Whittaker functions}\label{DWF}
In this section we follow the formulation in \cite[Section 3]{IY} or  \cite[Section 3]{KY1}. 

\subsection{Degenerate principal series: non-archimedean case}\label{DWFnonarc} 
Let $p$ be a prime number. Let $K:=G(\Z_p)$ which is the maximal open compact subgroup of $G(\Q_p)$. 
For each quasi-character $\mu:\Q^\times_p\lra \C^\times$, we consider the 
degenerate principal series representation $I(\mu):={\rm Ind}^{G(\Q_p)}_{P(\Q_p)}\mu|\cdot|^s \circ \nu$. 
For each $z\in \C$ and $f\in I(\mu)$, we define $f^{(z)}\in I(\mu |\cdot |^z)$ by 
$$f^{(z)}(g)=|t|^z f(k),\ g=\diag(t,m,t^{-1})n(x) k\in G(\Q_p)=M(\Q_p)N(\Q_p)K.$$ 
Let $\psi:\Q_p\lra \C^\times$ be the standard additive character. For each $\eta\in V'(\Q_p)$ with $q(\eta)\neq 0$, 
let $\psi_\eta:N(\Q_p)\lra \C^\times,\ n(x)\mapsto \psi((\eta,n(x)))$. 
For each $f^{(z)}\in I(\mu |\cdot |^z)$, we define the Jacquet integral 
$$\bw^{\mu,z}_\eta(f^{(z)}):=\ds\int_{N(\Q_p)}f^{(z)}(w_0 n(x))\overline{\psi(\eta,x)}dx$$
where $w_0$ is the Weyl element defined in Section \ref{pre}. If $f^{(z)}$ is fixed by $\diag(1,-I_{n+2},1)\in K$, 
by change of variables ($x\mapsto -x$), 
$\bw^{\mu,z}_\eta(f^{(z)})=\ds\int_{N(\Q_p)}f^{(z)}(w_0 n(x))\psi(\eta,x)dx$ which matches with the setting in Section \ref{AFSO}. 
If we write $\mu=\mu'|\cdot |^{s(\mu)}$ uniquely with $\mu'$ unitary and $s(\mu)\in \R$, then 
$\bw^{\mu,z}_\eta(f^{(z)})$ is absolutely convergent in $z$ satisfying ${\rm Re}(z)+s(\mu)>\frac{n+2}{2}$. The computation in the previous section shows 
it is a polynomial in $p^{-z}$. Therefore, we can evaluate $\bw^{\mu}_\eta(f^{(z)})$ at $z=0$ and put 
$\bw^\mu_\eta:=\bw^{\mu,0}_\eta$. According to \cite[theorems in Section 4]{KY2}, 
for $s(\mu)>-\frac{1}{2}$ if $n$ is odd, and for $s(\mu)>-1$ if $n$ is even,  
we normalize $\bw^\mu_\eta$ by  
$$w^\mu_\eta(f)=|q(\eta)|^{\frac{n+2}{4}}\bw^\mu_\eta(f)\times 
\begin{cases}
L(1,\mu) & \text{if $n$ is even},\\
\ds\frac{L(1,\mu^2)}{L(\frac{1}{2},\chi_\eta \mu)} & \text{if $n$ is odd}
\end{cases}.$$
Under the above assumption on $s(\mu)$, if $\mu$ is unramified, by \cite[theorems in Section 4]{KY2}, we have 
\begin{equation}\label{recoverQ}
\bw^\mu_\eta(f_p)=|q(\eta)|^{\frac{n+2}{4}}Q_{\eta,p}(\mu(p))
\end{equation}
for the unramified section $f_p \in I(\mu)$ defined as in \cite[Section 4]{KY2} or \cite[Section 4]{Po}.   
 
We summarize the basic properties of $I(\mu)$ and $w^\mu_\eta$. 
They can be proved in a similar way to those in \cite[Section 3]{IY} and \cite[Section 3]{KY1}. 
\begin{prop}\label{propertiesIW} Let $\eta\in V'(\Q_p)$ with $q(\eta)\neq 0$. Let $\mu:\Q_p^\times\lra \C^\times$ be a continuous character. Then, the following holds:
\begin{enumerate}
\item If $\mu$ is unramified, $I(\mu)$ admits a unique spherical component 
of which is denoted by $A(\mu)$.  
\item \cite[Theorem 4.1]{Jantzen} Assume $n$ is odd. Then, $I(\mu)$ is reducible if and only if 
\begin{enumerate}
\item $s=\frac{n+2}{2}, \frac{1}{2}$ and $\mu=1$, or
\item $s=\frac{1}{2},\ \mu^2=1$ and  $\mu\neq 1$.
\end{enumerate}
\item \cite[Theorem 5.5]{BanJ} Assume $n$ is even. Then, $I(\mu)$ is reducible if and only if 
\begin{enumerate}
\item $s=\frac{n+2}{2}, 1$ and $\mu=1$, or
\item $s=0,\ \mu^2=1$ and  $\mu\neq 1$.
\end{enumerate}
\item The Whittaker space ${\rm Wh}_\eta(I(\mu)):={\rm Hom}_{N(\Q_p)}(I(\mu),\psi_\eta)$ is of one-dimensional. 
\item Assume that $s(\mu)>-1$ if $n$ is even, and $s(\mu)>-\frac{1}{2}$ if $n$ is odd. Then $w^\mu_\eta$ is a generator of ${\rm Wh}_\eta(I(\mu))$ and 
it satisfies
 $$w^\mu_\eta(\diag(t,m,t^{-1})n(x)\cdot f)=\psi_{t\cdot\eta\cdot {}^t m^{-1}}(x)\mu(t)^{-1}w^\mu_{t \cdot\eta\cdot {}^t m^{-1}}(f)$$
 for $f\in I(\mu)$,\ $\diag(t,m,t^{-1})n(x)\in P(\Q_p)=M(\Q_p)N(\Q_p)$. 
 \item Assume that $\mu$ is an unramified character and that $s(\mu)>-1$ if $n$ is even, and $s(\mu)>-\frac{1}{2}$ if $n$ is odd. Then, the restriction of  
 $w^\mu_\eta$ to $A(\mu)$ is non-zero. In particular, it gives a generator of ${\rm Wh}_\eta(A(\mu))$. 
\end{enumerate}
\end{prop}
\begin{proof}The first claim follows from ${\rm dim}I(\mu)^K=1$. 
The second claim follows from \cite[Theorem 4.1, page 32,]{Jantzen} for odd $n$. 
The third claim follows from \cite[Proposition 5.1, p. 467, Theorem 5.5, p.475]{BanJ} for 
even $n$. 
The fourth claim follows from \cite[Theorem 3.2]{Karel}.  
 The fifth claim follows from the previous claims. The transformation law is easy to check as in the proof of \cite[Lemma 3.4(2)]{KY1}. 

For the final claim, we first note that $M(\Q_p)$ acts transitively on the set $\{\eta\in V'(\Q_p)\ |\ q(\eta)\neq 0\}$. 
By the transformation law, we may check the claim for a specific $\eta$ and a spherical section similarly defined as in \cite[Section 4]{KY2}. Then, explicit examples in \cite[Section 4]{KY2} (the case $p=2$ is similarly handled) show  
the non-triviality of the Jacquet integrals.  
\end{proof}

\subsection{Degenerate principal series: archimedean case}\label{DWFarc} 
We keep the notation in Section \ref{AFSO}. Let $K$ be the maximal compact subgroup of $G(\R)$. 
Let $\psi=\psi_\infty:\R\lra \C^\times,\ x\mapsto e^{2\pi \sqrt{-1}x}$ be the standard additive character. 
Let $l\ge 0$ be the integer and $(\tau_l,\mathbb{V}_l={\rm Sym}^{2l}(\C^2))$ the algebraic representation of $K$.  
For $\eta\in V(\R)'$ with $q(\eta)> 0$, recall the Pollack's spherical function (\ref{PollackSpherical}):
$$\W^{(l)}_\eta(g)=q(\eta)^{\frac{l+1}{2}}\psi((\eta,x))\tau_l(k^{-1})\W_\eta(t,m),\ g=n(x)\diag(t,m,t^{-1})k\in G(\R)=N(\R)M(\R)K.$$
By direct computation, we see that 
$$\W^{(l)}_\eta(n(x')\diag(t',m,t'^{-1}) n(x)\diag(t,m,t^{-1}))=\psi((\eta,x'+t'\cdot x \cdot {}^t m'))\W_{t' \cdot\eta\cdot {}^t m'^{-1}}(t,m)$$
for  $n(x')\diag(t',m,t'^{-1})\in P(\R)=N(\R)M(\R)$.

\section{The Weil representation of the Jacobi group $J$}\label{weil-rep}
We refer to \cite{Ik94} for the basics of the Weil representation. 
Let $\psi:\Q\bs \A\lra \C^\times$ be the standard additive character (cf \cite[Section 4.2]{KY}).
For each $S\in \Q^\times$, let $\psi_S=\psi(S\ast )=\otimes'_{p}\psi_{S,p}$ where we also regard 
it with an additive character on $Z(\A)$ via the identity $Z(\A)\stackrel{\sim}{\lra} \A,\ v(0,0,z)\mapsto z$. 

Let ${\rm Mp}_2(\A)$ be the metaplectic covering of $H(\A)$. 
Let $\omega_{\psi_S}:=\otimes'\omega_{S,p}=
\omega_{S,\f}\otimes \omega_{S,\infty}:
\widetilde{J(\A)}:={\rm Mp}_2(\A)\ltimes U(\A)\lra {\rm Aut}_\C(\mathcal{S}(X(\A)))$ be the Weil representation 
associated to $\psi_S$ acting on the Schwartz space $\mathcal{S}(X(\A))$ for $X(\A)$. 
Explicitly, for each place $p\le \infty$ and $\Phi=\otimes'_{p\le\infty}\Phi_p\in \mathcal{S}(X(\A))$, it is given by, 
for $\ve\in\{\pm 1\}$, $v(x,y,z)\in U(\Q_p),\ t\in X(\Q_p),\ a\in \Q^\times_p$, and $b\in \Q_p$,   
\begin{enumerate}
\item $\omega_{S,p}(v(x,y,z))\Phi_p(t)=\Phi_p(t+x)
\psi_{S,p}\Big(z+\sigma(t,y)+\frac{1}{2}\sigma(x,y)\Big).$
\item 
$\omega_{S,p}((\ell(\begin{pmatrix}
 a & 0 \\
 0 & a^{-1}
 \end{pmatrix}),\ve))\Phi_p(t)=\ve^n\ds\frac{\gamma_p(1)}{\gamma_p(S a^n)}|a|^{\frac{n}{2}}\Phi_p(at)$,  
 \item 
$\omega_{S,p}(\ell(\begin{pmatrix}
 1 & b \\
 0 & 1
 \end{pmatrix}),\ve))\Phi_p(t)=\ve^n
\psi_{S,p}(\sigma(t,t)b)\Phi_p(t)  ,\ b\in \Q_p,\ t\in X(\Q_p)$. 
\item 
$
\omega_{S,p}(\ell(\left(\begin{array}{cc}
 0 & -1   \\ 
1 & 0 
\end{array}\right)),\ve))\Phi_p(t)=\ve^n(F_S\Phi_p)(-t),\ t\in X(\Q_p)$,  
$$(F_S\Phi_p)(t)=\ds\int_{X(\Q_p)} \Phi_p(x)\psi_{S,p}(\sigma(t,x))dx$$ 
\end{enumerate}
where   $\gamma_p(a)$ stands for the Weil constant  associated to $\psi_{S,p}$and $dx$ means the Haar measure on $X(\Q_p)$ which is self-dual with respect to the Fourier transform $F_S$.  

Recall that $X(\R)=V''(\R)$ is a quadratic space with signature $(p,q)=(1,n)$.
By direct computation, it is easy to see that the $\widetilde{\SO(2)}$-type of 
$\omega_{S,\infty}$ (for $S\in \R^\times$) is given by  
\begin{equation}\label{Ktype}
\sgn(S)\Big(\frac{p-q}{2}+m\Big)=\sgn(S)\Big(\frac{1-n}{2}+m\Big),\ m\in \Z_{\ge 0}. 
\end{equation}

For each $\Phi\in 
\mathcal{S}(X(\A))$, we define the theta function 
\begin{eqnarray}\label{theta}
\Theta_{\psi_S}(v(x,y,z)\widetilde{h};\Phi)&:=&\sum_{\xi\in X(\Q)}
\omega_{\psi_S}(v(x,y,z)\widetilde{h})\Phi(\xi),\ 
v(x,y,z)\in U(\Bbb A),\ \widetilde{h}\in \widetilde{H}(\Bbb A) \\
&=&\sum_{\xi\in X(\Q)}(\omega_{\psi_S}(\widetilde{h})\Phi)(x+\xi)\psi(S\sigma(\xi,y))
\psi\Big(S(z+\frac{1}{2}\sigma(x,y))\Big).\nonumber 
\end{eqnarray}

\section{Fourier-Jacobi expansions of the Eisenstein series along $J$: the general case}\label{FJES}
\label{generalcase}
In this section we will prove that any Fourier-Jacobi coefficient of the Eisenstein series 
are the Eisenstein series. We refer to \cite{Ik94} for basics of this topic. 

\subsection{Double coset decomposition}\label{dcd}

Since $n\geq 3$, the relative root system of $G$ over $\Q$ is of type $B_3$, and we denote the positive roots by
$\{e_1\pm e_2, e_1\pm e_3, e_2\pm e_3, e_1, e_2, e_3\}$, and let $\Delta=\{e_1-e_2, e_2-e_3, e_3\}$ be the set of simple roots, where $e_i\pm e_j$ has multiplicity one, and $e_i$ has multiplicity $8m$.

The double coset $P\backslash G/Q$ is bijective to the double coset of the Weyl group $W_P\backslash W_G/W_Q$. By \cite{Ca}, page 64,
each double coset of $W_P\backslash W_G/W_Q$ has unique element of minimal length, and they are
$\{1, (1 2 3), (1 2)c_2\}$, where $c_i$ is the Weyl group element attached to $e_i$, and $(i j)$ is the Weyl group element attached to $e_i-e_j$. Then $G=P\xi_2 Q\cup P\xi_1 Q\cup P\xi_0 Q$, and $P\xi_0 Q$ is the unique open cell, where $\xi_2=1$, $\xi_1=(1 2 3)$, and $\xi_0=(1 2)c_2$. 

In terms of group elements, we can take coset representatives as 
$$\xi_1=
\left(\begin{array}{ccccccc}
0 & 0 & 1 & 0& 0& 0 & 0  \\ 
1 & 0 & 0  & 0&0& 0 & 0  \\ 
0 & 1 & 0  & 0&0& 0 & 0  \\
0 & 0 &0& I_{n-2} &0  & 0 & 0  \\ 
0 & 0 & 0  &0 & 0 & 1 & 0 \\ 
0 & 0 & 0  &0 & 0 & 0 & 1 \\ 
0 & 0 &0 &0 & 1 & 0 & 0 
\end{array}\right),\quad
\xi_0=
\left(\begin{array}{ccccccc}
0 & 0 & 0 & 0& 0& 1 & 0  \\ 
1 & 0 & 0  & 0&0& 0 & 0  \\ 
0 & 0 & 1  & 0&0& 0 & 0  \\
0 & 0 &0& I_{n-2} &0  & 0 & 0  \\ 
0 & 0 & 0  &0 & 1 & 0 & 0 \\ 
0 & 0 & 0  &0 & 0 & 0 & 1 \\ 
0 & 1 &0 &0 & 0 & 0 & 0 
\end{array}\right).
$$
By direct computation, we have $\xi_2^{-1} P\xi_2\cap U=U$, 
$$\xi_1^{-1} P\xi_1\cap U=
\left\{\left(\begin{array}{ccccc}
1 & 0  & -\textbf{y}J_{1,n-1} & \ast & \ast \\
 0 & 1  &-\textbf{x}J_{1,n-1} & \ast &  \ast \\
  0 & 0 & 1_n & {}^t\textbf{x} & {}^t\textbf{y} \\
  0 & 0 & 0 &  1 & 1 \\
   0 & 0 &  0& 0& 1 \\
\end{array}
\right)\in U\ \Bigg|\ 
\begin{array}{c}
\textbf{x} =(0,x_2,\ldots,x_n),\\
 \textbf{y} =(0,y_2,\ldots,y_n)
 \end{array}
 \right\},
$$
and 
$$\xi_0^{-1} P\xi_0\cap U=
\left\{\left(\begin{array}{ccccc}
 1& \ast  & -\textbf{y}J_{1,n-1} &0 & \ast \\
 0 &1  &0& 0 &  0 \\
  0 & 0 & 1_n &0 & {}^t\textbf{y} \\
  0 & 0 & 0 &  1& \ast \\
   0 & 0 &  0& 0& 1 \\
\end{array}
\right)\in U\ \Bigg|\ \textbf{y}\in \mathbb{G}^n_a\right\}.
$$
As shown in \cite[Lemma 7.3]{KY}, we have 
\begin{equation}\label{ppxi0q}
P(\Q)\bs P(\Q)\xi_0 Q(\Q)=\xi_0 (Y(\Q)\bs U(\Q))\cdot 
(B_{H^{{\rm ss}}}(\Q)\bs H^{{\rm ss}}(\Q)) 
\end{equation}
where $ B_{H^{{\rm ss}}}$ is the Borel subgroup of $H^{{\rm ss}}$ corresponding to 
the upper Borel subgroup of $\SL_2$ via the map (\ref{mapl}). 

\subsection{Fourier-Jacobi expansions of Eisenstein series}
We identify $H$ with $SL_2$ via the map (\ref{mapl}). Let 
$B$ is the Borel subgroup of $SL_2$ which consists of upper-triangular matrices.
Let $\omega$ be a unitary character of $\Q^\times\backslash\A^\times$ and $s\in \C$. 
Let $\mathbb{K}=\SL_2(\widehat{\Z})\times \SO(2)$ be the standard maximal compact subgroup of $\SL_2(\A)$. Let $\widetilde{\mathbb{K}}$ be the pullback of $\mathbb{K}$ 
to ${\rm Mp}_2(\A)$ under the projection ${\rm Mp}_2(\A)\lra \SL_2(\A)$. 

We denote by $I(\omega,s)$, the degenerate principal series representation of $G(\A)$ consisting of any  
function $f:G(\A)\lra \C$ such that 
$$f(pg)=\delta^{\frac{1}{2}}_P(p)|\nu(p)|^s_\A\omega(\nu(p))f(g)$$
 for any  $p\in P(\A)$ and any  $g\in G(\A)$.  
We also define  the space $I_1(\omega,s)$ consisting of any smooth, $\mathbb{K}$-finite 
function $f:{\rm SL}_2(\A)\lra \C$ such that 
$$f(pg)=\delta^{\frac{1}{2}}_B(p)|a|^s_{\A}\omega(a)f(g)$$
 for any  $p=
\left(\begin{array}{cc}
a & b \\
0 & a^{-1} 
\end{array}
\right)\in B(\A)$ and any  $g\in {\rm SL}_2(\A)$. 
Here   
$\delta^{\frac{1}{2}}_B(p)=|a|_\A$ for $p=
\left(\begin{array}{cc}
a & b \\
0 & a^{-1} 
\end{array}
\right)\in B(\A)
$.    
Similarly we also define  the space $\widetilde{I}_1(\omega,s)$ consisting of any smooth, 
$\widetilde{\mathbb{K}}$-finite function $f:{\rm Mp}_2(\A)\lra \C$ such that 
\begin{equation}\label{ActionOfvarepsilon}
f(\widetilde{p} g)=\ve \frac{\gamma(1)}{\gamma(a)}\delta^{\frac{1}{2}}_B(p)|a|^s_{\A}\omega(a)f(g)
\end{equation}
 for any  $\widetilde{p}=(p,\ve)\in \widetilde{B(\A)}$ with $p=\left(\begin{array}{cc}
a & b \\
0 & a^{-1} 
\end{array}
\right)$ and any  $g\in {\rm Mp}_2(\A)$. 
Here $\gamma:=\prod_p \gamma_p$ is the global Weil constant (with respect to $\psi$) and $\ve$ in the RHS of (\ref{ActionOfvarepsilon}) is 
understood as multiplying $+1$ if $\ve \in \{(t_v)_v\in \bigoplus_v \{\pm 1\}\ |\ \prod_v t_v=1 \}$, $-1$ otherwise.

For any section $f\in I(\omega,s)$, we define the Eisenstein series on $G(\A)$ of type $(\omega,s)$ by 
$$E(g,s;f):=\sum_{P(\Q)\ba G(\Q)}f(\gamma g),\ g\in G(\A).$$
For each $S\in \Q^\times$, 
let 
$$E_S(g;f)=\int_{Z(\Q)\bs Z(\A)}E(zg,s;f)\overline{\psi_S(z)}dz,\ g\in G(\A).$$
For each $S\in \Q^\times$ and $\Phi\in \mathcal{S}(X(\A))$, put 
\begin{equation}\label{Fourier-Jacobi}
E_{\psi_S,\Phi}(\widetilde{h};f):=\int_{U(\Q)\backslash U(\A)} E_S(vh;f)
\overline{\Theta_{\psi_S}(v\widetilde{h};\Phi)}\, dv,\ \widetilde{h}=(h,\ve)\in {\rm Mp}_2(\A).
\end{equation}
When ${\rm dim}\hspace{0.2mm}X=n$ is even, ${\rm Mp}_2(\A)$ is canonically isomorphic to 
$\SL_2(\A)\times \{\pm 1\}$ and thus, we can define 
\begin{equation}\label{Fourier-Jacobi-even}
E_{\psi_S,\Phi}(h;f):=\int_{U(\Q)\backslash U(\A)} E_S(vh;f)
\overline{\Theta_{\psi_S}(vh;\Phi)}\, dv,\ h\in \SL_2(\A).
\end{equation}
Let $\iota_{\SL_2}=\left(\begin{array}{cc}
0 & -1 \\
1 &0
\end{array}
\right)$ be the Weyl element of $\SL_2$. 
The following theorem is proved by a standard way (see the proof of \cite[p.627, Theorem 3.2]{Ik94}) with the double coset decomposition in Section \ref{dcd}. Thus, we omit the proof.
\begin{theorem}\label{EisenEisen}Keep the notation above. 
Let  $\chi_S(\ast)=
\langle -S ,\ast   \rangle_\A$ is the character of $\Q^\times\bs\A^\times$ defined by the Hilbert symbol 
$\langle \ast ,\ast   \rangle_\A$,
Assume that 
$\Phi\in \mathcal{S}(X(\A))$ is $\widetilde{\mathbb{K}}$-finite.  
For ${\rm Re}(s)\gg 0$,   
\begin{enumerate}
\item when ${\rm dim}\hspace{0.2mm}X=n$ is even, 
$R(h;f,\Phi):=\ds\int_{\bold V(\A)}f(\xi_0 \cdot v\cdot \iota_{\SL_2}\cdot h)\overline{\omega_S(v(\iota_{SL_2}\cdot h))\Phi(0)}dv$ is a section of $I_1(\omega \chi_S,s)$, 
\item when ${\rm dim}\hspace{0.2mm}X=n$ is odd, 
$R(\widetilde{h};f,\Phi):=\ds\int_{\bold V(\A)}f(\xi_0 \cdot v\cdot \iota_{SL_2}\cdot 
\widetilde{h})\overline{\omega_S(v(\iota_{SL_2}\cdot \widetilde{h}))\Phi(0)}dv$ is a section of $\widetilde{I}_1(\omega \chi_S,s)$  where  
\item $E_{\psi_S,\Phi}(\ast;f)$ is an Eisenstein series on 
$\SL_2(\A)$ or ${\rm Mp}_2(\A)$ associated to $R(\ast;f,\Phi)$ 
according to ${\rm dim}\hspace{0.2mm}X=n$ is even or odd.  
\end{enumerate}
\end{theorem}

\section{Fourier-Jacobi expansions of $E_l(g,\Phi_{\f})$ along $J$: an explicit form}\label{FJEl}
Recall the Eisenstein series $E_l(g,\Phi_{f}):=E_l(g,\Phi_{\f},l+1)$ 
defined by the Pollack's section $f_l(\ast,\Phi_{\f},l+1)\in I(\mathbf{1},l-\frac{n}{2})=
{\rm Ind}^{G(\A)}_{P(\A)}|\nu|^{l-\frac{n}{2}}$ (see Section \ref{AFSO}). 
In this section, we explicitly compute the Fourier-Jacobi coefficients of $E_l(g,\Phi_{\f})$. 

\subsection{Preliminaries at the archimedean place}\label{prearchi}
For each $a,S\in \Q$, put $\eta:=(a,\overbrace{0,\ldots,0}^n,S)\in V'(\Q)$ and 
we view it as $$\eta=(0,a,\overbrace{0,\ldots,0}^n,S,0)\in V(\Q).$$ 
We assume that  
\begin{equation}\label{PC}
(\eta,\eta)_{V'}:=\frac{1}{2}\eta J_{2,n}{}^t \eta=aS>0.
\end{equation}
Hence $a$ and $S$ have the same sign. 
 Notice that the action of $m\in \SO(V)$ on $\textbf{x}\in V$ is defined by 
$$m \textbf{x}:=\textbf{x}\cdot {}^t m$$ as a matrix multiplication. 
Further, for any $\textbf{x},\ \textbf{y}\in V$ and $m\in \SO(V)$, 
$$(\textbf{x},m\textbf{y})_V=\frac{1}{2}\textbf{x}J_{3,n+1}({}^t(\textbf{y}\cdot {}^t m)=
\frac{1}{2}\textbf{x}(J_{3,n+1} m){}^t\textbf{y}=\frac{1}{2}\textbf{x}({}^t m^{-1}J_{3,n+1}){}^t\textbf{y}=
(m^{-1}\textbf{x},\textbf{y})_V.$$

In the computation later, we need to consider 
\begin{equation}\label{trans1}
v(x,0,0)\ell(\diag(t,t^{-1}))=\diag(t,m,t^{-1})\in M(\R), m:=v(x,0,0)\ell(\diag(1,t^{-1}))
\end{equation}\label{need}
for $t\in GL_1$ and $x=(x_1,x',x_n)\in X$ where $x_1,x_n\in \mathbb{G}_a$ 
and $x'\in \mathbb{G}^{n-2}_a$. For such $(t,m)\in M(\R)$, we compute   
\begin{eqnarray}\label{ueta}
u_{2\pi \eta}(t,m)&=& i t (2\pi\eta,m(i v_1-v_2)) \nonumber \\
&=&- 2\pi i \{(a t^2+S)-i\cdot S t(x_1+x_n)-St^2 x_1x_n+St^2(x',x')\}.  
\end{eqnarray}
where $(x',x')=\frac{1}{2}x'A{}^t x'$ and $i=\sqrt{-1}$ and $v_1,v_2$ are given in (\ref{vv}). 

For above $(t,m)\in M(\R)$ with $m:=v(x,0,0)\ell(\diag(1,t^{-1}))$ in (\ref{trans1}), we will also compute 
\begin{equation}\label{Iab}
I_{\eta}(t,\alpha,\beta;\Phi):=\int_{X(\R)}\mathcal{W}_{2\pi \eta}(t,m)
\overline{F_{\alpha,\beta}(x;\Phi)}dx,\ t,\alpha,\beta\in \R,\ t>0,\beta>0,\ \Phi\in 
\mathcal{S}(X(\R)) 
\end{equation}
where 
\begin{eqnarray}\label{fab}
F_{\alpha,\beta}(x;\Phi)&=&\sqrt{\beta}^{\frac{n}{2}}\Phi(\sqrt{\beta} x) 
e^{2\pi \sqrt{-1}S \frac{\alpha}{2}\sigma(x,x)},\ \Phi \in S(X(\R))  
\end{eqnarray}
and 
$$\mathcal{W}_{2\pi \eta}(t,m)=t^{l+1}
\sum_{v=-l}^l \Bigg(\frac{|u_{2\pi \eta}(t,m)|}{u_{2\pi \eta}(t,m)}\Bigg)^v 
K_v (|u_{2\pi \eta}(t,m)|)X^{l-v}Y^{l+v}$$
where $u_{2\pi \eta}(t,m)$ is given explicitly in (\ref{ueta}). 
Here the exponent $\frac{n}{2}=\frac{1}{2}{\rm dim} X$ in $F_{\alpha,\beta}(x;\Phi)$ comes from 
the action of the Weil representation. 
Note that by using $K_v=K_{-v}$, we can also write it as 
$$\W_{2\pi \eta}(t,m)=t^{l+1}
\sum_{v=-l}^l \Bigg(\frac{u_{2\pi \eta}(t,m)}{|u_{2\pi \eta}(t,m)|}\Bigg)^v 
K_v (|u_{2\pi \eta}(t,m)|)X^{l+v}Y^{l-v}.
$$
We write (\ref{Iab}) as 
\begin{equation}\label{IabII}
I_{\eta}(t,\alpha,\beta;\Phi)=\sum_{v=-l}^l I_{\eta}(t,\alpha,\beta;\Phi)_v X^{l+v}Y^{l-v}
\end{equation}
where 
\begin{equation}\label{Iabv}
I_{\eta}(t,\alpha,\beta;\Phi)_v:=\int_{X(\R)}
t^{l+1}\Bigg(\frac{u_{2\pi \eta}(t,m)}{|u_{2\pi \eta}(t,m)|}\Bigg)^v 
K_v (|u_{2\pi \eta}(t,m)|)
\overline{F_{\alpha,\beta}(x;\Phi)}dx
\end{equation}
for above $(t,m)\in M(\R)$.

Changing $(x_1,x_n)$ with $(t_1+t_n,t_1-t_n)$ with the Jacobian $2$, the equation 
(\ref{ueta}) becomes  
\begin{eqnarray}\label{ueta2}
u_{2\pi \eta}(t,m)&=& i\cdot 
\sgn(S)\{(\lambda\cdot t_1+t^{-1}\lambda \cdot i)^2-\mu\}
\end{eqnarray}
where 
$$\lambda:=t\sqrt{2\pi |S|},\ \mu:=2\pi t^2 \Big(|a|+|S| t^2_n+|S|(x',x')\Big).$$
Henceforth we assume $t=\sqrt{\beta}$ for $\beta>0$ so that 
$$\lambda=\sqrt{\beta}\sqrt{2\pi |S|},\ 
\mu=2\pi \beta \Big(|a|+|S| t^2_n+|S|(x',x')\Big).
$$ 
Note that $\lambda,\mu>0$ since $aS,\beta>0$ and $(x',x')\ge 0$. 
For each $S\in \Q^\times$ and $r\in \Z_{\ge 0}$, put 
\begin{equation}\label{Schwartz1}
\Phi^1_{S,r}(t_1)=(2\pi |S|)^{\frac{r}{2}}  t^r_1 e^{-2\pi |S| t^2_1}
\end{equation}
 so that 
$\Phi^1_{S,r}(\sqrt{\beta} t_1)=(\lambda t_1)^r e^{-\lambda^2 t^2_1}$ 
 and 
 \begin{equation}\label{Schwartz2}
 \Phi_{S,r}:=
\Phi^1_{S,r}\otimes\Phi'\in S(X(\R)) 
\end{equation}
 where $\Phi'\in S(\R^{n-1})$ with respect to the 
coordinate $(x',t_n)\in \R^{n-1}\subset V''(\R)=X(\R)$.  We remark that the restriction of 
the quadratic space $V''$ to $\R^{n-1}$ is negative definite. 
Applying Theorem \ref{Po-const} and Remark \ref{rodd} to the case when $v=0$, we have that  
\begin{eqnarray}\label{Izerob}
&& I_{\eta}(\sqrt{\beta},0,\beta,\Phi_{S,r})_0=2\beta^{\frac{l+1}{2}+\frac{n}{4}}  \times \nonumber \\
&& \int_{(x',t_n)\in \R^{n-1}}\Bigg(\int_{t_1\in \R}
(\lambda t_1)^r e^{-\lambda^2 t^2_1} 
K_0(|(\lambda\cdot t_1+\sqrt{\beta}^{-1}\lambda \cdot i)^2-\mu|) dt_1
\Bigg)\overline{\Phi'(\sqrt{\beta} x',\sqrt{\beta} t_n)}dx'dt_n \nonumber \\
&=&2\beta^{\frac{l+1}{2}+\frac{n}{4}}\lambda^{-1}C(S,r)
\Bigg(\int_{(x',t_n)\in \R^{n-1}} e^{-\mu}\overline{\Phi'(\sqrt{\beta} x',\sqrt{\beta} t_n)}  dx'dt_n\Bigg)  \nonumber \\
&=&2\beta^{\frac{l+1}{2}+\frac{n}{4}}\lambda^{-1}C(S,r) C(S,\beta,\Phi') e^{-2\pi \beta |a|}  \nonumber
\end{eqnarray}
where 
\begin{equation}\label{C(S)}
C(S,r):=
\left\{\begin{array}{ll}
0 & \text{if $r$ is odd,} \\
C(r,\sqrt{\beta}^{-1}\lambda)=\ds\int_{t\in \R} t^r e^{-t^2}K_0(t^2+2\pi |S|)dt 
& \text{if $r$ is even,}
\end{array}\right.
\end{equation}
is the constant 
given in Theorem \ref{Po-const} for $r$ even (we use the notation $C(r,\sqrt{\beta}^{-1}\lambda)$ therein)
and 
$$C(S,\beta,\Phi'):=\int_{(x',t_n)\in \R^{n-1}} e^{-2\pi\beta |S| \{t_n^2+(x',x')\}}
\overline{\Phi'(\sqrt{\beta} x',\sqrt{\beta} t_n)}  dx'd_{t_n}.$$
By the change of variables, we have 
$$C(S,\beta,\Phi')=C_1(S,\Phi')\beta^{-\frac{n-1}{2}},\ 
 C_1(S,\Phi'):=C(S,1,\Phi').
$$
Thus, we have 
\begin{equation}\label{simpleform}
I_{\eta}(\sqrt{\beta},0,\beta,\Phi_{S,r})_0=C_2(S,r) 
\sqrt{\beta}^{l-\frac{n}{2}+1}e^{-2\pi |a| \beta }
\end{equation}
where $C_2(S,\Phi_{S,r}) =\ds\frac{2C(S,r)C_1(S,\Phi')}{\sqrt{2\pi|S|}}$. 

\begin{remark}\label{non-vanishing} Let $S\in \Q^\times$. If $\Phi_{S,0}$ is 
Gaussian, the constant $C_2(S,\Phi_{S,0})$ is non-zero. 
\end{remark}

Since $\{\Phi^1_{S,r}\}_{r\ge 0}$ is a topological basis of $\mathcal{S}(\R)$, we have the following result:
\begin{theorem}\label{completeform}Let $S\in \Q^\times$. For any $\Phi\in \mathcal{S}(X(\R))$, 
there exists a constant $C_3(S,\Phi)$ depending only on $S$ and $\Phi$ such that 
$$I_{\eta}(\sqrt{\beta},0,\beta,\Phi_{S})_0=C_3(S,\Phi) 
\sqrt{\beta}^{l-\frac{n}{2}+1}e^{-2\pi |a| \beta }.
$$
\end{theorem}

Similarly, for any $v$-th component, we have 
\begin{theorem}\label{v-comp}Let $S\in \Q^\times$. For even $r\ge 0$ and $v$ with 
$-l\le v\le l$,  $$I_{\eta}(\sqrt{\beta},0,\beta,\Phi_{S,r})_v=C_2(S,\Phi_{S,r})(-i)^v \sgn(S)^v   
\sqrt{\beta}^{l-\frac{n}{2}+1}e^{-2\pi |a| \beta }.
$$
\end{theorem}

\subsection{Computation of Fourier-Jacobi coefficients}\label{cfjc}
Now, we are ready to compute the Fourier-Jacobi expansion of $E_l$. 
Recall the settings in Section \ref{weil-rep}. 

Let $S\in \Q^\times$ and $\Phi=\Phi_{\f}\otimes \Phi_\infty\in S(X(\A))$. 
We define the Fourier-Jacobi coefficient of $E_l(\ast,\Phi_{\f},l+1)$ with respect to $\Phi$ and of index $S$ by 
\begin{equation}\label{FJC}
E_l(\widetilde{h})_{S,\Phi}:=\int_{U(\Bbb Q)\bs U(\Bbb A)}E_l(v(x,y,z)h,\Phi_{\f},l+1)\overline{\Theta_{\psi_S}
(v(x,y,z)\widetilde{h};\Phi)}dv(x,y,z),  \widetilde{h}\in \widetilde{H(\Bbb A)}
\end{equation} 
where $\widetilde{h}=(h,\ve)\in \widetilde{H(\Bbb A)}$. 
Notice that $n(x)=v(0,0,z)$ for $x=(z,\overbrace{0,\ldots,0}^{n+1})\in V'$ 
and for such $x$ and $\eta=(\eta_1,\ldots,\eta_{n+2})\in V'(\Bbb Q)$, we have 
$\psi_{\eta}(n(x))=\psi(\tfrac{1}{2}\eta_{n+2}z)$. 
Thus, it is easy to see that 
$$E_l(\widetilde{h})_{S,\Phi}=\int_{U(\Bbb Q)\bs U(\Bbb A)}E_l(v(x,y,z)h,\Phi_{\f},l+1)_{S}\overline{\Theta_{\psi_S}
(v(x,y,z)\widetilde{h};\Phi)}dv(x,y,z), $$
where 
$$E_l(g,\Phi_{\f},l+1)_{S}=\sum_{\eta=(\ast,\ldots,\ast,2S)\in V'(\Q)\atop q(\eta)\ge 0}
a_{E_l}(\eta)(g_{\f}\cdot \Phi_{\f})\mathcal{W}_{2\pi \eta}(g_\infty),\ g=g_{\f}g_\infty\in G(\A).$$
It is easy to see that 
$v(x,y,z)=v(0,y,z+(x,y)_{V''})v(x,0,0)$ where $(x,y)_{V''}=\frac{1}{2}xJ_{1,n-1}{}^t y=\frac{1}{2}\sigma(x,y)$. Then, for $\eta=(\eta_1,\eta',2S)\in V'(\Q)$ with 
$\eta'\in V''(\Q),\ \eta_1\in \Q$, 
$$\psi_{\eta}(v(0,y,z+(x,y)_{V''}))=
\psi(\tfrac{1}{2}\sigma(\eta',y))\psi(S(z+\tfrac{1}{2}\sigma(x,y))).$$

Then, we have 
$$E_l(\widetilde{h})_{S,\Phi}=\int_{X(\Q)\bs X(\A)}\Bigg(
\sum_{\eta=(a,\eta',2S)\in V'(\Q)\atop q(\eta)\ge 0}
a_{E_l}(\eta)(v(x_{\f},0,0)h_{\f}\cdot \Phi_{\f})\mathcal{W}_{2\pi \eta}(v(x_\infty,0,0)h_\infty)
\Bigg) $$
$$\times \overline{\sum_{\xi\in X(\Q)}(\omega_{\psi_S}(h)\Phi)(x+\xi)}
\Bigg(\int_{Y(\Q)\bs Y(\Q)}\psi(\sigma(\tfrac{\eta'}{2}-S\xi,y))dy\Bigg) dx.
$$
We observe that the integral 
$\ds\int_{Y(\Q)\bs Y(\Q)}\psi(\sigma(\tfrac{\eta'}{2}-S\xi,y))dy$ is zero unless 
$\eta'=2S\xi$. Further, we can write 
\begin{equation}\label{etaxi}
(a,2S\xi,2S)=v(-\xi,0,0)\cdot \eta_{a,S,\xi}\cdot v(\xi,0,0),\ 
\eta_{a,S,\xi}=(a+S\sigma(\xi,\xi),\overbrace{0,\ldots,0}^n,S).
\end{equation}
Thus, by unfolding the sum $\ds\sum_{\xi\in X(\Q)}$ in the theta function and 
$\ds\int_{X(\Q)\bs X(\A)}$, we have 
\begin{equation}\label{fj-exp1}
E_l(\widetilde{h})_{S,\Phi}=\ds\sum_{\xi\in X(\Q)=V''(\Q)}E_{\xi,S,\Phi}(\widetilde{h})
\end{equation}
where 
$$E_{\xi,S,\Phi}(\widetilde{h})=\sum_{a\in \Q \atop 
aS+\frac{S^2}{2}\sigma(\xi,\xi)\ge 0}
\int_{X(\A_f)}a_{E_l}(\eta_{a,S,\xi})(v(x_{\f},0,0)\widetilde{h}_{\f}\cdot \Phi_{\f})
\overline{(\omega_{\psi_{S,\f}}(v(x_{\f},0,0)\widetilde{h}_{\f})\Phi_{\f})(0)}dx_{\f} $$
$$\times \int_{X(\R)}\mathcal{W}_{2\pi \eta_{a,S,\xi}}(v(x_\infty,0,0)\widetilde{h}_\infty)\overline{
(\omega_{\psi_{S,\infty}}(v(x_\infty,0,0)\widetilde{h}_\infty)\Phi_{\infty})(0)}dx_\infty.$$
Note that (\ref{fj-exp1}) is a finite sum (see \cite[Section 5.2]{KY}).

Thus we have 
$$E_{\xi,S,\Phi}(\widetilde{h})=\sum_{a\in \Q \atop 
aS+\frac{S^2}{2}\sigma(\xi,\xi)\ge 0}a_{E_l}(\eta_{a,S,\xi})(\widetilde{h}_{\f}\Phi_{\f})I_{\eta}(a,S,\xi;\Phi_\infty)
(\widetilde{h}_\infty),\ \widetilde{h}=\widetilde{h}_{\f}\widetilde{h}_\infty\in H(\A)$$
where 
\begin{equation}\label{inf-part1}I_{\eta}(a,S,\xi;\Phi_\infty)(\widetilde{h}_\infty):=
\int_{X(\R)}\mathcal{W}_{2\pi \eta_{a,S,\xi}}(v(x_\infty,0,0)\widetilde{h}_\infty)\overline{
(\omega_{\psi_{S,\infty}}(v(x_\infty,0,0)\widetilde{h}_\infty)\Phi_{\infty})(0)}dx_\infty.
\end{equation}

Next, we compute $I_{\eta}(a,S,\xi;\Phi_\infty)(\widetilde{h}_\infty)$ with 
$\widetilde{h}_\infty=(h_\infty,\ve_\infty)\in \widetilde{\SL_2(\R)}$. By Iwasawa decomposition of 
$\SL_2(\R)$, 
we may compute it for $h_\infty=
\begin{pmatrix}
1 & \alpha \\
0 & 1 
\end{pmatrix}\diag(\sqrt{\beta},\sqrt{\beta}^{-1})
$, $\alpha\in \R,\ \beta\in \R_{>0}$ so that $h_\infty(\sqrt{-1})=\alpha+\beta\sqrt{-1}\in \mathbb{H}$. 
It is easy to see that 
$$v(x_\infty,0,0)h_\infty=\Bigg(v(x_\infty,0,0)\begin{pmatrix}
1 & \alpha \\
0 & 1 
\end{pmatrix}v(-x_\infty,0,0)\Bigg)v(x_\infty,0,0)\diag(\sqrt{\beta},\sqrt{\beta}^{-1})$$
$$\phantom{xxxxxxxxxxxx}=n(-\alpha,\alpha x_\infty,\frac{\alpha}{2}
\sigma(x_\infty,x_\infty))v(x_\infty,0,0)\diag(\sqrt{\beta},\sqrt{\beta}^{-1})\in N(\R)M(\R)$$
and 
$$v(x_\infty,0,0)\diag(\sqrt{\beta},\sqrt{\beta}^{-1})=\diag(\sqrt{\beta},m,\sqrt{\beta}^{-1}),\ 
m=v(x_\infty,0,0)\ell(1,\sqrt{\beta}^{-1}).$$
Using this, we have 
$$\mathcal{W}_{2\pi \eta_{a,S,\xi}}(v(x_\infty,0,0)h_\infty)=
\psi_{w_1(2\pi \eta_{a,S,\xi})w^{-1}_1}(n(-\alpha,\alpha x_\infty,\tfrac{\alpha}{2}
\sigma(x_\infty,x_\infty)))$$
$$\phantom{xxxxxxxxxxxxxxx}\times \mathcal{W}_{2\pi \eta_{a,S,\xi}}(v(x_\infty,0,0)\diag(\sqrt{\beta},\sqrt{\beta}^{-1}))$$
$$=\psi_\infty(-\alpha(a+S\sigma(\xi,\xi)))\psi_\infty(\tfrac{S\alpha}{2}\sigma(x_\infty,x_\infty))
\mathcal{W}_{2\pi \eta_{a,S,\xi}}(\beta,m)$$
while by definition of the Weil representation and recalling (\ref{fab}), 
\begin{eqnarray*}
\omega_{\psi_{S,\infty}}(v(x_\infty,0,0)\widetilde{h}_\infty)\Phi_{\infty})(0) &=&
\ve^n_\infty \frac{\gamma_\infty(1)}{\gamma_\infty(S)}
F_{\alpha,\beta}(x_\infty;\Phi_\infty) \\
& =& \ve^n_\infty  \frac{\gamma_\infty(1)}{\gamma_\infty(S)}
\beta^{\frac{n}{4}}\Phi_\infty(\sqrt{\beta} x_\infty)\psi_\infty(\tfrac{S\alpha}{2}\sigma(x_\infty,x_\infty)).
\end{eqnarray*}
Therefore, we have 
$$I_{\eta}(a,S,\xi;\Phi_\infty)(h_\infty)=\ve^n_\infty 
\frac{\gamma_\infty(-1)}{\gamma_\infty(-S)}e^{-2\pi\alpha \sqrt{-1}(a+S\sigma(\xi,\xi))}
\int_{X(\R)}\mathcal{W}_{2\pi \eta_{a,S,\xi}}(\beta,m)
\overline{F_{0,\beta}(x;\Phi_\infty)}dx_\infty$$
$$=\ve^n_\infty  \frac{\gamma_\infty(-1)}{\gamma_\infty(-S)} e^{-2\pi\alpha \sqrt{-1}(a+S\sigma(\xi,\xi))}I_{\eta_{a,S,\xi}}(\sqrt{\beta},0,\beta,\Phi_\infty)$$
in terms of the notation (\ref{Iab}). Summing up, we have 
\begin{prop}\label{fjexp2} 
Keep the notation being as above. It holds that 
\begin{equation}\label{fjexp-eq1}
E_{\xi,S,\Phi}(\widetilde{h})=\ve^n_\infty  \frac{\gamma_\infty(-1)}{\gamma_\infty(-S)}\sum_{a\in \Q \atop 
aS+S^2\sigma(\xi,\xi)\ge 0}a_{E_l}(\eta_{a,S,\xi})(h_{\f}\Phi_{\f})
I_{\eta_{a,S,\xi}}(\sqrt{\beta},0,\beta,\Phi_\infty)e^{-2\pi\alpha \sqrt{-1}(a+S\sigma(\xi,\xi))}
\end{equation}
for $\widetilde{h}=(h,\ve)\in {\rm Mp}_2(\A)$ with $h_\infty \sqrt{-1}=\alpha+\beta \sqrt{-1}$.  
\end{prop}
According to (\ref{IabII}), we write 
\begin{equation}\label{vsum}
E_{\xi,S,\Phi}(\widetilde{h})=\sum_{v=-l}^l E_{\xi,S,\Phi}(\widetilde{h})_v X^{l-v}Y^{l+v}
\end{equation}
and $E_{\xi,S,\Phi}(\widetilde{h})_v $ is said to be the $v$-th component of $E_{\xi,S,\Phi}(\widetilde{h})$. 
Combining this proposition and applying (\ref{simpleform}) and Theorem \ref{completeform}, \ref{v-comp} for $\eta_{a,S,\xi}$ 
(just replace $a$ with $a+S\sigma(\xi,\xi)$ therein), we have 
\begin{corollary}\label{explicitformula1}Keep the notation being as above. 
Let $\Phi=\Phi_\f\otimes \Phi_\infty \in \mathcal{S}(X(\A))$. If $v\neq 0$, suppose that 
$\Phi_\infty=\Phi_{S,r}$ with even $r\ge 0$. Then,  
if $S>0$, 
$$
E_{\xi,S,\Phi}(\widetilde{h})_v=C_4(S,\Phi_\infty) \ve^n_\infty 
\sqrt{\beta}^{l-\frac{n}{2}+1} 
\sum_{a\in \Q_{>0} \atop 
a+S\sigma(\xi,\xi)\ge 0}a_{E_l}(\eta_{a,S,\xi})(h_{\f}\Phi_{\f})
(\bar{q}_\tau)^{a+S\sigma(\xi,\xi)}
$$
and if $S<0$, 
$$
E_{\xi,S,\Phi}(\widetilde{h})_0=C_4(S,\Phi_\infty) \ve^n_\infty  \sqrt{\beta}^{l-\frac{n}{2}+1} 
 \sum_{a\in \Q_{<0} \atop 
a+S\sigma(\xi,\xi)< 0}a_{E_l}(\eta_{a,S,\xi})(h_{\f}\Phi_{\f})(q_\tau)^{-(a+S\sigma(\xi,\xi))} 
$$
where $C_4(S,\Phi_\infty) =\frac{\gamma_\infty(-1)}{\gamma_\infty(-S)}\times 
\left\{\begin{array}{ll}
C_3(S,\Phi_\infty)  & \text{if $v=0$} \\
C_2(S,\Phi_{S,r}) & \text{if $v\neq 0$} 
\end{array}\right.
$ and $q_\tau=e^{2\pi\sqrt{-1} \tau}$ with $\tau=\alpha+\beta \sqrt{-1}$. 
\end{corollary}

\begin{remark}\label{partial}
For each $\Phi=\Phi_{\f}\otimes \Phi_\infty\in \mathcal{S}(X(\A))$ and $\xi\in X(\R)$, 
there exist $\Phi^{(1)}_{\xi,\f},\ldots,\Phi^{(r)}_{\xi,\f}\in \mathcal{S}(X(\A_\f))$ and 
constants $c_1,\ldots,c_r\in \C$ such that 
$$E_{\xi,S,\Phi}=\sum_{i=1}^r c_i E_{S,\Phi^{(i)}_{\xi,\f}\otimes\Phi_\infty}$$ 
{\rm(}this follows from the argument for \cite[Section 5.2]{KY} and 
the proof of \cite[Lemma 5.4]{KY1}{\rm)}. 

The same is true for each $v$-th component $(E_{\xi,S,\Phi})_v$. 
\end{remark}

We denote by $\kappa_\theta\in \widetilde{\SO(2)}$ corresponding to $e^{i\theta }\in U(1)$ under 
the natural identification $\widetilde{\SO(2)}\simeq \SO(2)\simeq U(1)$. 
Let $\Phi_\infty\in \mathcal{S}(X(\R))$ be an $\widetilde{\SO(2)}$-eigen function of $S(X(\R))$ for $\omega_{S,\infty}$. 
It follows from (\ref{Ktype}) that 
there exists an integer $k\in \Z_{\ge 0}$ such that 
\begin{equation}\label{eigenphi}
\omega_{S,\infty}(\kappa_\theta)\Phi_\infty=
e^{i\theta\cdot\sgn(S)(\frac{2-n}{2}+k)}\Phi_\infty,\ \kappa_\theta\in \widetilde{\SO(2)}.
\end{equation}
Thus, $\Phi_\infty$ is of $\widetilde{\SO(2)}$-weight $\sgn(S)(\frac{2-n}{2}+k)$. 
The bottom case of $k=0$ is attained by the Gaussian on $X(\R)$. 

The following result is an immediate consequence of Theorem \ref{EisenEisen}  with 
Remark \ref{partial} and it is consistent with Corollary \ref{explicitformula1}:
\begin{corollary}\label{explicitformula2}Keep the notation being as above.  
Let $\chi_S=\langle -S,\ast \rangle_\infty$ is the character of $\A^\times$ defined by the 
global  Hilbert symbol and $S$.
Let $\Phi=\Phi_\f\otimes \Phi_\infty\in \mathcal{S}(X(\A))$ and assume 
$\Phi_\infty$ is of $\widetilde{\SO(2)}$-weight {\rm(}-type{\rm)} $\sgn(S)(\frac{2-n}{2}+k)$. 
Then, the $v$-th component $(E_{\xi,S,\Phi})_v$ for each $-l \le v\le l$ is the Eisenstein series defined by a section of 
$\left\{\begin{array}{cc}
I_1(\chi_S,l-\frac{n}{2})  & \text{if $n$ is even} \\
 \widetilde{I}_1(\chi_S,l-\frac{n}{2})  & \text{if $n$ is odd}
\end{array}\right.
$
 whose
$\left\{\begin{array}{cc}
\text{$\SO(2)$-type}  & \text{if $n$ is even} \\
\text{$\widetilde{\SO(2)}$-type}   & \text{if $n$ is odd}
\end{array}\right.
$
is $v-\sgn(S)\Big(\frac{2-n}{2}+k\Big).$
In particular, $E_{\xi,S,\Phi}(\widetilde{h})_v$ corresponds to  
$$\begin{cases}
\text{a holomorphic modular form in $h_\infty \sqrt{-1}$ of weight $l-\frac{n}{2}+1$}, &\text{if $S<0$ and $v+k=l$}\\
\text{an anti-holomorphic modular form in $h_\infty \sqrt{-1}$ of weight $-(l-\frac{n}{2}+1)$}, &\text{if $S>0$ and $v-k=-l$}.
\end{cases}
$$
\end{corollary}

\begin{remark}\label{ClaAde} For any modular form $f:\mathbb{H}\lra \C$ of weight $k (\in \Z_{\ge 0})$ with respect to a suitable 
congruence subgroup, one can associate an automorphic form $\varphi_f$ on $\SL_2(\A)$ so that $\varphi_f((1_{\f},h_\infty))=
j(h_\infty,\sqrt{-1})^k f(h_\infty \sqrt{-1})$.  
Conversely, any automorphic form $\varphi$ on $\SL_2(\A)$ for which  
$j(h_\infty,\sqrt{-1})^{-k}\varphi((h_{\f},h_\infty))$ is holomorphic for every fixed $h_{\f}\in \SL_2(\A_{\f})$ arises from a 
modular form of weight $k$ in the above manner.  The last statement of Corollary \ref{explicitformula2} follows this sense.  
\end{remark}

\section{Fourier-Jacobi expansion in terms of degenerate Whittaker functions}\label{FJE-DWF}
In this section, we interpret the computation of Section \ref{FJES}, \ref{FJEl} in terms of degenerate Whittaker functions 
as in \cite[Section 7]{IY} or \cite[Section 5]{KY1}. 
We freely use the notation in Section \ref{weil-rep}. 
For $F\in\{\Q_p,\R,\A,\A_f\}$,  put
$\oH(F)=\left\{\begin{array}{ll}
\SL_2(F) & \text{for $n$ even,}\\ 
{\rm Mp}_2(F)  & \text{for $n$ odd}
\end{array}\right.
$ for simplicity. 
Let $B_{\oH}$ be the standard Borel subgroup of $\oH$ and $N_{\oH}$ be the unipotent radical of $B_{\oH}$. 
Let $w_{\oH}$ be the Weyl element of $\oH$.
 
\subsection{The non-archimedean case} 
Let $p$ be a prime number and $\mu$ a character of $\Q^\times_p$ with $s(\mu)>
\left\{\begin{array}{ll}
-1 & \text{for $n$ even,}\\ 
\frac{-1}{2}  & \text{for $n$ odd}
\end{array}\right.$. Let $\psi=\psi_p$ be the standard additive character of $\Q_p$ and 
let 
$\overline{I}(\mu)={\rm Ind}^{\oH(\Q_p)}_{B_{\oH}(\Q_p)}\mu$ denote the normalized parabolic induction of $\mu $ to $\oH$ where $\overline{I}(\mu)$ is defined similarly as in 
(\ref{ActionOfvarepsilon}) with respect to $\psi$ when $n$ is odd.  
For $S\in \Q^\times_p$, let $\chi_S=\langle -S,\ast \rangle$ be the character of $\Q^\times_p$ corresponding to 
the extension $\Q_p(\sqrt{-S})/\Q_p$ via the local class field theory. 
For each $f\in I(\mu)={\rm Ind}^{G(\Q_p)}_{P(\Q_p)}\mu\circ \nu$ with $s(\mu)>-\frac{1}{2}$, $\Phi\in \cS(X(\Q_p))$, and $h\in \oH(\Q_p)$, we define 
$$\beta^\psi_S(h;f\otimes\Phi):=\int_{Y(\Q_p)}\int_{Z(\Q_p)}f(w_0 w_{\oH} v(y,0,z)h)\overline{\omega_{S,p}(v(y,0,z)\Phi)(0)}dzdy\times 
\left\{\begin{array}{ll}
1 & \text{for $n$ even,}\\ 
\frac{1}{L(\frac{1}{2},\mu\chi_S)}  & \text{for $n$ odd}
\end{array}\right..
$$
\begin{lemma}\label{localFJnonarchi}Keep the notation and assumption as above. Then, 
$$\beta^\psi_{S}:I(\mu)\otimes_\C \cS(X(\Q_p))\lra \I(\mu \chi_S)$$
is a $V(\Q_p)$-invariant and $H(\Q_p)$-intertwining, non-zero map. Hence, 
$\beta^\psi_S(v\cdot h;f\otimes \omega_{S,p}(\gamma)\Phi)=\I(\mu)(\gamma)
\beta^\psi_S(h;f\otimes \Phi)$ for any $v\in V(\Q_p)$ and $h\in \oH(\Q_p)$.  
\end{lemma}
\begin{proof}It is proved in a similar way to \cite[Lemma 5.1]{KY1}. By choosing a Schwartz function with a 
sufficiently small support, one can check the map is non-zero.
\end{proof}

For $\phi\in \I(\mu)$ with $s(\mu)>-1$ if $n$ even, $s(\mu)>-\frac{1}{2}$ if $n$ odd  and $a\in \Q^\times_p$, define
\begin{equation}\label{dwfsl2}
w^{\mu}_a(\phi_p)=|a|^{\frac{1}{2}}_p \int_{N_{\oH}(\Q_p)}\phi\Big(
w_{\oH}\left(\begin{array}{cc}
1 & z \\
0 & 1 
\end{array}
\right) \Big)\overline{\psi(az)}dz 
\times 
\left\{\begin{array}{ll}
L(1,\mu) & \text{for $n$ even,}\\ 
\frac{L(1,\mu^2)}{L(\frac{1}{2},\mu \chi_a)}  & \text{for $n$ odd}
\end{array}\right..
\end{equation}

\begin{lemma}\label{explocalFJ} Keep the notation and assumption as above. 
For $a\in \Q^\times_p$, $f\in I(\mu)$, and $\Phi\in \cS(X(\Q_p))$, it holds 
$$w^\mu_a(\beta^\psi_S(\ast;f\otimes\Phi))=|aS|^{-\frac{n+2}{4}}|a|^{\frac{1}{2}}
\int_{X(\Q_p)}\overline{\Phi(x)}w^\mu_{(a,\mathbf{0},S)}(v(x,0,0)\cdot\Phi)dx$$
\end{lemma}
\begin{proof}It is proved in a similar way to \cite[Lemma 5.2]{KY1}. 
\end{proof}

\subsection{The archimedean case} 
Let $\psi(x)=e^{2\pi i x}$ be the standard additive character. For each $a\in \R^\times$, 
put $\psi_a:=\psi(a\ast)$. 
Let $ \overline{\SO}(2)$ be $\SO(2)$ if $n$ even,  the double covering of $\SO(2)$ if $n$ odd.  
For $\oH(\R)$, we define the degenerate Whittaker functions with respect to $\psi_a$ as follows. 
For $\ell\in \frac{1}{2}\Z$ and $a\in \R^\times $, define 
$$W^{(\ell)}_a(g):=
\left\{\begin{array}{l}
|a|^{\frac{\ell}{2}}\exp(2\pi\sqrt{-1}a\tau)j^{\frac{1}{2}}_{\oH(\R)}(g,\sqrt{-1})^{-2\ell},\ {\rm if}\  g=\gamma\in 
\oH(\R)=\SL_2(\R) \\ 
\ve |a|^{\frac{\ell}{2}}\exp(2\pi\sqrt{-1}a\tau)j^{\frac{1}{2}}_{\oH(\R)}(g,\sqrt{-1})^{-2\ell},\ {\rm if}\  g=(\gamma,\ve)\in 
\oH(\R)={\rm Mp}_2(\R)
\end{array}\right.$$
where $\tau=\gamma\sqrt{-1}\in\mathbb{H}$ and $j^{\frac{1}{2}}_{\oH(\R)}(g,\sqrt{-1})$ is the canonical automorphy factor 
of weight $\frac{1}{2}$ for 
$\oH(\R)$  such that $(j^{\frac{1}{2}}_{\oH(\R)}(g,\sqrt{-1}))^2=j(\gamma,\sqrt{-1})$ is the usual automorphic factor of $\SL_2(\R)$. 
It is easy to see that 
\begin{equation}\label{sl2whi}
W^{(\ell)}_a\left(\left(\begin{array}{cc}
1& z'\\
0& 1 
\end{array}
\right)\left(\begin{array}{cc}
c& 0\\
0& c^{-1} 
\end{array}
\right)\cdot g\cdot k \right)=\exp(2\pi\sqrt{-1} az')W^{(\ell)}_{c^2 a}(g)j^{\frac{1}{2}}_{\oH(\R)}(k,\sqrt{-1})^{-2\ell}
\end{equation}
for $z'\in \R,\ c\in \R^\times, g\in \oH(\R)$, and $k\in \overline{\SO}_2(\R)$. 
If $\gamma=\left(\begin{array}{cc}
1& \alpha\\
0& 1 
\end{array}
\right)\left(\begin{array}{cc}
\sqrt{\beta}& 0\\
0& \sqrt{\beta}^{-1} 
\end{array}
\right)$ with $\alpha\in \R$ and $\beta\in\R_{>0}$, then 
$W^{(\ell)}_a(\gamma)=|a|^{\frac{\ell}{2}}\sqrt{\beta}^\ell q^a_\tau$ when $\oH(\R)=\SL_2(\R)$ and  
$W^{(\ell)}_a((\gamma,\ve))=\ve |a|^{\frac{\ell}{2}}\sqrt{\beta}^\ell q^a_\tau$ when $\oH(\R)={\rm Mp}_2(\R))$ where 
$\tau=\gamma\sqrt{-1}$. 

The following claim directly follows from the results in Section \ref{prearchi}. 
\begin{lemma}\label{localFJarchi}Let $a,S\in \Q^\times$ with $aS>0$. 
Let $\Phi_{S,0}$ be a Schwartz function on $X(\R)$ defined in (\ref{Schwartz1}). Assume $\Phi_{S,0}$ is Gaussian. 
Then, it holds that 
$$\int_{X(\R)}\W^{(l)}_{(a,\mathbf{0},S)}(v(x,0,0)h)_0\overline{\omega_{S,\infty}(h)\Phi_{S,0}(x)}dx=C(S,\Phi_{S,0})|a|^{\frac{n}{4}}
\left\{\begin{array}{cc}
\overline{W}^{l-\frac{n}{2}+1}_a(h) & \text{if $S>0$,} \\
W^{l-\frac{n}{2}+}_a(h) & \text{if $S<0$ }
\end{array}\right.
$$
for $h\in \oH(\R)$ 
where $C(S,\Phi_{S,0})$ is a non-zero constant depending on $S$ and $\Phi_{S,0}$ and 
$\W^{(l)}_{(a,\mathbf{0},S)}(v(x,0,0)h)_0$ is the 0-th component of $\W^{(l)}_{(a,\mathbf{0},S)}(v(x,0,0)h)$. Further, $\overline{W}^{l-(\frac{n+2}{2})}_a(h)$ is the complex conjugation of 
$W^{l-(\frac{n+2}{2})}_a(h)$.
\end{lemma}

\begin{remark}\label{constC} The constant $C(S,\Phi_{S,0})$ in the above lemma can be explicitly computed once 
$\Phi_{S,0}$ is given explicitly. 
\end{remark}

\subsection{The global case}
Let $\Pi=\otimes'_p \Pi_p\otimes\Pi_\infty$ be an automorphic representation of $G(\A)$ with an intertwining map  
$\Pi\hookrightarrow \mathcal{A}(G(\Q)\bs G(\A)),\ \phi\mapsto F_\phi$ where the target is the space of automorphic functions on $G(\A)$. For each prime $p$, assume $\Pi_p$ is a unique unramified summand of ${\rm Ind}^{G(\Q_p)}_{P(\Q_p)}\mu^\kappa_p\circ \nu$. 
Assume that there is a non-zero vector 
$\phi=\otimes'_p\phi_p\otimes\phi_\infty=\phi_\f\otimes\phi_\infty\in \Pi$ such that  
$F_\phi$ is a modular form of weight $l$ on $G(\A)$ in the sense of Section \ref{AFSO} and 
the Fourier expansion  along $N$ take the form 
$$F_\phi=F_{\phi,0}+\sum_{0\neq \eta\in V'(\Q),\ q(\eta)>0}F_{\phi,\eta},\ F_{\phi,\eta}(g)=\ds\int_{N(\Q)\bs N(\A)}F_\phi(n(x)g)\overline{\psi(\eta,x)}dx$$
where $F_{\phi,0}$ is the constant term along $N$. 
Thus, we assume $F_\phi$ does not have the rank one Fourier coefficients. 
 
By multiplicity one of Whittaker spaces (Proposition \ref{propertiesIW} for non-archimedean places and \cite[Theorem 3.2.4]{Po} for 
archimedean place), for $\eta\in V'(\Q)$ with $q(\eta)>0$, there is a complex number $c_\eta(\phi)$ such that 
$$F_{\phi,\eta}(g)=c_\eta(\phi)w^{\mu^\kappa_\f}_\eta(g_\f\cdot \phi_\f)\W^{(l)}_\eta(g_\infty),\ 
w^{\mu^\kappa_\f}_\eta(g_\f\cdot \phi_\f):=\prod_p w^{\mu^\kappa_p}_\eta(g_p\cdot \phi_p)$$
for $g=g_\f g_\infty \in G(\A)$ with $g_\f=(g_p)_p$. 
By automorphy and the transformation law of the degenerate Whittaker functions, the collection $\{c_\eta(\phi)\}_{\eta\in V'(\Q),\ 
q(\eta)>0}$ satisfies 
\begin{equation}\label{formulaC}
c_{t\cdot \eta\cdot {}^t m^{-1}}(\phi)=\mu_\f(t)^{-\kappa}c_\eta(\phi),\ \diag(t,m,t)\in M(\Q).
\end{equation}
For $S\in \Q^\times$, let $(F_\phi)_{S}(g):=\ds\int_{Z(\Q)\bs Z(\A)}F_\phi(z g)\overline{\psi(Sz)}dz$. 
Then, it is easy to see that 
$$(F_\phi)_{S}=\sum_{\eta=(\ast,\ast,2S)\in V'(\Q),\ q(\eta)>0}c_\eta(\phi)w^{\mu^\kappa_\f}_\eta(g_\f\cdot \phi_\f)\W^{(l)}_\eta(g_\infty).$$

Let $\Phi=\Phi_\f\otimes \Phi_\infty\in \cS(X(\A))$ where $\Phi_\infty$ is that of Lemma \ref{localFJarchi}. 
Recall the Fourier-Jacobi coefficient of $F_\phi$ at $(S,\Phi)$ from Section \ref{cfjc}:
$$
F_\phi(h)_{S,\Phi}:=\int_{U(\Bbb Q)\bs U(\Bbb A)}(F_\phi)_S(v h)\overline{\Theta_{\psi_S}
(v h;\Phi)}dv,\ h\in\oH(\Bbb A).
$$
The following claim follows from previous lemmas with the argument for \cite[Lemma 5.4(1)]{KY1}
\begin{lemma}\label{FJglobalDWF}Keep the notation and assumption as above. 
Let $S\in \Q^\times$ and $\chi_S=\langle -S,\ast \rangle_{\A}=\otimes_p'\chi_{S,p}$. 
Then, $F_\phi(h)_{S,\Phi}=$ 
$$C_1(S,\Phi)\times 
\left\{\begin{array}{cc}
\ds\sum_{a\in \Q^\times_{>0}}c_{(a,\mathbf{0},S)}(\phi)w^{\mu^2_\f\chi_S}_a(\beta^{\psi_\f}_{S}(h_\f;\phi_\f\otimes\Phi_\f))
\overline{W}^{l-\frac{n}{2}+1}_a(h_\infty) & \text{if $n$ is even and $S>0$,} \\
\ds\sum_{a\in \Q^\times_{<0}}c_{(a,\mathbf{0},S)}(\phi)w^{\mu^2_\f\chi_S}_a(\beta^{\psi_\f}_{S}(h_\f;\phi_\f\otimes\Phi_\f))
W^{l\frac{n}{2}+1}_a(h_\infty) & \text{if $n$ is even and $S<0$,} \\
\ds\sum_{a\in \Q^\times_{>0}}c_{(a,\mathbf{0},S)}(\phi)w^{\mu_\f\chi_S}_a(\beta^{\psi_\f}_{S}(h_\f;\phi_\f\otimes\Phi_\f))
\overline{W}^{l-\frac{n}{2}+1}_a(h_\infty) & \text{if $n$ is odd and $S>0$,}  \\
\ds\sum_{a\in \Q^\times_{<0}}c_{(a,\mathbf{0},S)}(\phi)w^{\mu_\f\chi_S}_a(\beta^{\psi_\f}_{S}(h_\f;\phi_\f\otimes\Phi_\f))
W^{l-\frac{n}{2}+1}_{a}(h_\infty) & \text{if $n$ is odd and $S<0$}
\end{array}\right.
$$ 
for $h=h_\f h_\infty\in \oH(\A)$ 
where $C_1(S,\Phi)$ is a non-zero constant depending on $S$ and $\Phi$, and 
$$w^{\mu^\kappa_\f\chi_S}_a(\beta^{\psi_\f}_{S}(h_\f;\phi_\f\otimes\Phi_\f)):=\prod_p 
w^{\mu^\kappa_p\chi_{S,p}}_a(\beta^{\psi_p}_{S}(h_p;\phi_p\otimes\Phi_p)).$$  
\end{lemma}

\section{Proofs for the main theorems}\label{proof}
In this section, we prove Theorems \ref{ITC} and \ref{ITC-odd} but we only give a proof for 
$n$ even. The other case is similarly handled. 

Assume $n$ is even. 
We apply the results in the previous sections to the special case $S=-1$. 
Let $f$ be the newform of weight $k=l-\frac{n}{2}+1$ in the claim and $\widetilde{\phi}=\otimes'_p \widetilde{\phi}_p \otimes \widetilde{\phi}_\infty=\widetilde{\phi}_{\f} \otimes \widetilde{\phi}_\infty:\SL_2(\A)\lra \C$ be the cuspidal automorphic form attached to $f$. 
Let $\pi=\otimes'_p \pi_p\otimes \pi_\infty$ be the $\psi$-generic cuspidal representation of 
$\SL_2(\A)$ attached to 
$\widetilde{\phi}$. For each prime $p$, $\pi_p$ is a unique $\psi_{-1,p}$-generic spherical component of 
$\overline{I}(\mu^2_p)={\rm Ind}^{\SL_2(\Q_p)}_{B(\Q_p)}\mu_p^2$ for 
an unramified character $\mu_p:\Q^\times_p\lra \C^\times$. 
Put $\mu_\f=
\otimes'_p \mu_p $. 
As in \cite[p.580-581]{KY1}, we can expand $\widetilde{\phi}$ as 
$$\widetilde{\phi}(h)=\sum_{a\in \Q_{>0}}a^{\frac{k}{2}}c_{a}(f)
w^{\mu^2_\f}_a(\widetilde{\phi}_\f)W^{(k)}_a(h_\infty)$$
for $h=\gamma (k\times g_\infty)\in \SL_2(\A)=\SL_2(\Q)(\SL_2(\widehat{\Z})
\times \SL_2(\R))$. 
Applying Lemma \ref{localFJnonarchi} to $S=-1$ so that $\chi_S=1$, $\beta^\psi_{-1}$ is surjective onto $\pi_p$. Thus, for each $\widetilde{\phi}_p$, 
one can find a unique unramified vector $\phi_p\in I(\mu_p)$ such that $\phi_p(1)=1$  
and $\Phi_p\in \Phi_p(X(\Q_p))$ such that $\beta^\psi_{-1}(\widetilde{\phi}_p\otimes \Phi_p)=\widetilde{\phi}_p$. 
Set $\phi=\otimes'_p \phi_p\otimes \phi_\infty$ with $ \phi_\infty=\phi^{{\rm Po}}$ and $\Phi=\otimes_p \Phi_p\otimes \Phi_\infty$ 
with $\Phi_\infty=\Phi_{1,0}$ is chosen as in Lemma \ref{localFJarchi}. 
Let $F_\phi\in \Pi_f\subset L^2(G(\Q)\bs G(\A))$ for above $\phi$. Then, by Lemma \ref{FJglobalDWF}, 
$$F_\phi(h)_{-1,\Phi}=C_1(-1,\Phi)\sum_{a\in \Q_{<0}}c_{(a,\mathbf{0},-1)}(\phi)
w^{\mu^2_\f}_a(h_\f\cdot\widetilde{\phi}_\f)W^{(k)}_a(h_\infty),\ h=h_\f h_\infty\in \SL_2(\A)$$
is a cuspidal automorphic form which corresponding to an anti-holomorphic cusp form of 
weight $k$ (since $S=-1<0$ and $a<0$ ).
By construction $F_\phi(\ast)_{-1,\Phi}$ generates a $\psi$-generic 
cuspidal representation of $\SL_2(\A)$. 
Since $F_\phi(\ast)_{-1,\Phi}$ and the complex conjugation of $\widetilde{\phi}$ generate the same $\psi$-generic cuspidal representation of $\SL_2(\A)$, by strong multiplicity one 
(see \cite{ChaiZ} with \cite{Z}),  there exists a non-zero constant $C_2$ such that 
$c_{(a,\mathbf{0},-1)}(\phi)=C_2\times (-a)^{\frac{k}{2}}c_{-a}(f)$. We can absorb $C_2$ into 
$\phi_\f$ such that $c_{(a,\mathbf{0},-1)}(\phi)=(-a)^{\frac{k}{2}}c_{-a}(f)$. Note that $q((a,\mathbf{0},-1))=-a$. Thus, one can write 
$c_{(a,\mathbf{0},-1)}(\phi)=q((a,\mathbf{0},-1))^{\frac{k}{2}}c_{q((a,\mathbf{0},-1))}(f)$. 
For $\eta=(a,\mathbf{0},S)$ with $q(\eta)=aS>0$ and $S<0$. By (\ref{formulaC}) and the explicit formula 
$c_{-a}(f)=\mu_\f(-a^{-1})$ (see 
\cite[p.597, 606]{KY1}, but note that $\mu_p(p^{-1})=\alpha_p$ due to the convention), 
we have 
$$c_{\eta}(\phi)=\mu^{-2}_\f(-S)c_{(-\frac{a}{S},\mathbf{0},-1)}=\mu^{-1}_\f(S^2)
c_{\frac{a}{S}}(f)=c_{aS}(f)=c_{q(\eta)}.$$ 
Finally, for any $q\in V'(\Q)$ with $q(\eta)>0$, there exists $g=\diag(t,m,t^{-1})\in \SO(V')(\Q)$ with $t^2=1$ 
such that $g\eta=(a,0,S)$ 
for some $a,S\in \Q^\times$ with $aS>0$. Note that $aS=q(g\eta)=t^2q(\eta)=q(\eta)$. Thus, we have 
$$c_\eta(\phi)=c_{q(\eta)}(f).$$ 
By definition, it is easy to see that 
$$c_\eta(\phi)w_\eta^{\mu^2_\f}(\phi_\f)=A_f(1_\f)$$
as desired. 

\section{$L$-functions and Arthur parameters of theta lifts}\label{Lfunction}

Let $f=\ds\sum_{m=1}^\infty a_f(m)m^{\frac{k-1}{2}}q^m$ 
be a Hecke eigen cusp form of weight 
$k=\left\{\begin{array}{ll}
l-\frac{n}{2}+1 & \text{for $n$ even,}\\ 
2l-n+1  & \text{for $n$ odd}
\end{array}\right.
$
with respect to $\SL_2(\Z)$. Let $a_f(p)=\alpha_p+\alpha_p^{-1}$. 
Let $L(s,\pi_f)=\ds\prod_p L(s,\pi_{p})$, where $L(s,\pi_{p})=(1-\alpha_p p^{-s})^{-1}(1-\alpha_p^{-1}p^{-s})^{-1}$.

\subsection{$n$ even} 
Let $n=8a+2$ ($n>2$). 
Let $F=F_\phi$ be the Ikeda type lift of $f$ on $\SO(3,n+1)$ as in Theorem 1.1, and let $\Pi_f=\otimes' \Pi_p\otimes\Pi_\infty$ be the automorphic representation associated with $F$.

When $p$ is unramified, $G(\Q_p)\simeq \SO(4a+3,4a+3)$, and $\Pi_p={\rm Ind}_P^G \, |\ |^{2s_p}\rtimes 1_{\SO}$, where $p^{s_p}=\alpha_p$ and $1_{\SO}$ is the trivial representation of $\SO(4a+2,4a+2)$, i.e., the quotient of ${\rm Ind}_B^{\SO(4a+2,4a+2)}\, |\ |^{4a+1}\otimes |\ |^{4a}\otimes\cdots \otimes |\ |^0$, where $B$ is the Borel subgroup of $\SO(4a+2,4a+2)$. Notice $2s_p$ since we replace $p^{\frac {l-\frac n2}2}$ by $\alpha_p$.
If $n$ is odd, we replace $p^{l-\frac n2}$ by $\alpha_p$.

Consider $\widetilde G=\SO(4a+4,4a+4)$, and $R=M'N'$, $M'\simeq\GL_1\times \SO(4a+3,4a+3)$. Then Langlands-Shahidi theory applies to $(\widetilde G,M',\Pi_p)$, and we have
\begin{theorem}\label{standardL-even} The standard $L$-function of $\Pi_f$ is given by
$$L(s,\Pi_f)=L(s,{\rm Sym}^2\pi_f)
\prod_{i=-\frac{n}{2}}^{\frac{n}{2}}
\zeta(s+ i).
$$
\end{theorem}

\begin{remark} In \cite{LNP} Y. Li, N. Narita and A. Pitale constructed non-tempered cusp forms on
$O(1, 8n + 1)$ from Maass cusp forms. Their $L$-function has the exactly same shape.
\end{remark}

\begin{remark}  We give a conjectural Arthur parameter of $\Pi_F$: 
Let $\phi_f: \mathcal L\lra \SL_2(\Bbb C)$ be the hypothetical Langlands parameter attached to $f$. 
Then ${\rm Sym}^2\phi_f: \mathcal L\lra \SO(3,\Bbb C)$ parametrizes ${\rm Sym}^2\pi_f$.
The distinguished unipotent orbit $(8a+3)$ of $\SO(8a+3,\Bbb C)$ gives rise to a map
$\phi_u: \SL_2(\Bbb C)\lra \SO(8a+3,\Bbb C)$. Then consider
$$\phi={\rm Sym}^2\phi_f\oplus  \phi_u : \mathcal L\times \SL_2(\C)\lra \SO(3,\C)\times \SO(8a+3,\Bbb C)\subset \SO(8a+6,\Bbb C)={}^L G.
$$
We expect that $\phi$ parametrizes $\Pi_f$. The global A-parameter of $\Pi_f$ \cite{Wu} is
$${\rm Sym}^2\pi_f[1]\oplus 1[8a+3].
$$

\end{remark}

\subsection{$n$ odd} 
Let $n=8a+3$. 
Let $F=F_\phi$ be the Ikeda type lift of $f$ on $\SO(3,1+n)$ as in Theorem 1.3, and let $\Pi_f=\otimes' \Pi_p\otimes \Pi_\infty$ be the automorphic  representation associated with $F$.

When $p$ is unramified, $G(\Q_p)\simeq \SO(4a+3,4a+4)$, and $\Pi_p=\Ind_P^G \, |\ |^{s_p}\rtimes 1_{\SO}$, where $p^{s_p}=\alpha_p$ and $1_{\SO}$ is the trivial representation of $\SO(4a+2,4a+3)$, i.e., the quotient of ${\rm Ind}_B^{\SO(4a+2,4a+3)}\, |\ |^{4a+\frac 32}\otimes |\ |^{4a+\frac 12}\otimes\cdots \otimes |\ |^\frac 12.$

Consider $\widetilde G=\SO(4a+4,4a+5)$, and $R=M'N'$, $M'\simeq\GL_1\times \SO(4a+3,4a+4)$. Then Langlands-Shahidi theory applies to $(\widetilde G,M',\Pi_p)$, and we have
\begin{theorem}\label{standardL-odd} The standard $L$-function of $\Pi_F$ is given by
$$L(s,\Pi_f)=L(s,\pi_f)\prod_{i=0}^n \zeta\Big(s+\frac{n}{2}-i\Big)
$$
\end{theorem}

If $n=8a+3$, ${}^L G={\rm Sp}(8a+6,\Bbb C)$. The distinguished unipotent orbit $(8a+4)$ of ${\rm Sp}(8a+4,\Bbb C)$ gives rise to a map
$\phi_u: \SL_2(\Bbb C)\lra {\rm Sp}(8a+4,\Bbb C)$. Then we expect that
$$\phi=\phi_f\oplus  \phi_u : \mathcal L\times \SL_2(\C)\lra \SL_2(\C)\times {\rm Sp}(8a+4,\Bbb C)\subset {\rm Sp}(8a+6,\Bbb C),
$$
parametrizes $\Pi_f$. The global A-parameter of $\Pi_f$ \cite{Wu} is
$$\pi_f[1]\oplus 1[8a+4].
$$

\section{Theta lift as residual spectrum}\label{residual}

We noted in Section 4.1 that $\Pi_f$ is a non-zero, non-cuspidal, square integrable automorphic representation of $G(\A)$. Since it is not cuspidal, it should be in the residual spectrum. We make it more precise.

\subsection{$n$ even}
Let $n=8a+2$.
Let $\Pi_f^{{\rm ORS}}$ be the Oda-Rallis-Schiffmann lift of $f$ to $G'=\SO(2,n)$ \cite{Od,RS}.
Let $\Pi_f^{{\rm ORS}}=\otimes_p' \Pi_p^{{\rm ORS}}\otimes \Pi_\infty^{{\rm ORS}}$.
Then by \cite{LNP} or \cite{AG} (see also the computation in Section \ref{Theta}), the Satake parameter of $\Pi_p^{{\rm ORS}}$ is given by
$$\diag(p^{4a},...,p,1,\alpha_p^2,\alpha_p^{-2},1,p^{-1},...,p^{-4a}).
$$

Now let $G=\SO(3,n+1)$ and $P=MN$, $M\simeq \GL_1\times \SO(2,n)$, and
consider the global induced representation
$$I(s,\Pi_f^{{\rm ORS}})=\Ind_P^G |\ |^s\otimes \Pi_f^{{\rm ORS}}.
$$
One can form the Eisenstein series $E(s,f_s)$ for a section $f_s\in I(s,\Pi_f)$.
It has a simple pole at $s=\frac n2$, and the residue generates the residual automorphic representation
$$J\left(\frac n2,\Pi_f^{{\rm ORS}}\right)=\otimes_p' J\left(\frac n2,\Pi_p^{{\rm ORS}}\right)\otimes J\left(\frac n2,\Pi_\infty^{{\rm ORS}}\right).
$$
Then $J\left(\frac n2,\Pi_f^{{\rm ORS}}\right)$ is the theta lift $\Pi_f$ to $\SO(3,n+1)$.
The Satake parameter of $J\left(\frac n2,\Pi_p^{{\rm ORS}}\right)$ is
$$\diag(p^{4a+1},p^{4a},...,p,1,\alpha_p^2,\alpha_p^{-2},1,p^{-1},...,p^{-4a},p^{-4a-1}),
$$
so that the $L$-function is
$$L(s,\Sym^2\pi_f)\prod_{i=-\frac{n}{2}}^{\frac{n}{2}}\zeta(s+i)$$
which agrees with the result in Section \ref{Lfunction}.
As we show below that at the infinity place, 
$J\left(\frac n2,\Pi_\infty^{{\rm ORS}}\right)$ 
is isomorphic to  a unique irreducible component of $\Pi^{3,n+1}_{l,0}|_{G(\R)}$ 
containing $\mathcal{H}^{2l}(\C^3)\boxtimes \mathbf{1}$ as a minimal $K$-type. 

We use a short-hand notation that $\Ind_{\GL_1\times \SO(2,n)}^{\SO(3,n+1)}$ means $\Ind_P^G$, where $P=MN, M\simeq \GL_1\times \SO(2,n)$, and $G=\SO(3,n+1)$. By \cite{Kobayashi}, $\Pi_{l,0}^{3,n+1}$ is a unique subrepresentation of $\Ind_{\GL_1\times \SO(2,n)}^{\SO(3,n+1)} |\ |^{l}\otimes {\rm tr}_{\SO(2,n)}$. Now, $J(\frac n2,\Pi_\infty)$ is a unique subrepresentation of $\Ind_{\GL_1\times \SO(2,n)}^{\SO(3,n+1)} |\ |^{-\frac n2}\otimes \Pi_\infty$.
Here we used that the trivial representation ${\rm tr}_{\SO(2,n)}$ is a subrepresentation of $\Ind_{\GL_1\times \GL_1\times \SO(n-2)}^{\SO(2,n)} |\ |^{-\frac n2}\otimes |\ |^{-\frac n2+1}\otimes {\rm tr}_{\SO(n-2)}$.

Hence $\Pi_{l,0}^{3,n+1}$ is a subrepresentation of 
$$\Ind_{\GL_1\times \GL_1\times \GL_1\times \SO(n-2)}^{\SO(3,n+1)} |\ |^{l}\otimes |\ |^{-\frac n2}\otimes |\ |^{-\frac n2+1}\otimes {\rm tr}_{SO(n-2)},
$$
which is isomorphic to 

$$\Ind_{\GL_1\times \GL_1\times \GL_1\times \SO(n-2)}^{\SO(3,n+1)} |\ |^{-\frac n2}\otimes |\ |^{l}\otimes |\ |^{-\frac n2+1}\otimes tr_{\SO(n-2)}.
$$
On the other hand, by \cite{Kobayashi}, $\Pi_\infty$ is a subrepresentation of 
$$\Ind_{GL_1\times \SO(1,n-1)}^{\SO(2,n)} |\ |^{l}\otimes {\rm tr}_{\SO(1,n-1)},
$$
which is a subrepresentation of 
$$\Ind_{\GL_1\times \GL_1\times \SO(n-2)}^{\SO(2,n)} |\ |^{l}\otimes |\ |^{-\frac n2+1}\otimes {\rm tr}_{\SO(n-2)}.
$$
Hence $J(\frac n2,\Pi_\infty)$ is a subrepresentation of 

$$\Ind_{\GL_1\times \GL_1\times \GL_1\times \SO(n-2)}^{\SO(3,n+1)} |\ |^{-\frac n2}\otimes |\ |^{l}\otimes |\ |^{-\frac n2+1}\otimes tr_{\SO(n-2)}.
$$
Therefore, $J(\frac n2,\Pi_\infty)$  is a unique irreducible component of $\Pi^{3,n+1}_{l,0}|_{G(\R)}$ 
containing $\mathcal{H}^{2l}(\C^3)\boxtimes \mathbf{1}$ as a minimal $K$-type.

\subsection{$n$ odd} Let $n$ be odd, i.e., $n=8a+3$. Then the Satake parameter of $\Pi_p$ is

$$\diag(p^{4a+\frac 12},...,p^\frac 12,\alpha_p,\alpha_p^{-1},...,p^{-4a-\frac 12}).
$$
Now 
consider the global induced representation
$$I(s,\Pi_f^{{\rm ORS}})=\Ind_P^G |\ |^s\otimes \Pi_f^{{\rm ORS}}.
$$
One can form the Eisenstein series $E(s,f_s)$ for a section $f_s\in I(s,\Pi_f)$.
It has a pole at $s=\frac n2$, and the residue generates the residual automorphic representation
$$J\left(\frac n2,\Pi_f^{{\rm ORS}}\right)=\otimes_p' J\left(\frac n2,\Pi_p^{{\rm ORS}}\right)\otimes J\left(\frac n2,\Pi_\infty^{{\rm ORS}}\right).
$$
Then $J(\frac n2,\Pi_f)$ is the theta lift $\Pi_f$ of $h$ to $\SO(3,n+1)$, where $h$ is the weight $\frac {k+1}2$ holomorphic form corresponding to $f$ under the Shimura correspondence.
The Satake parameter of $J(\frac n2,\Pi_p)$ is
$$\diag(p^{4a+\frac 32},...,p^\frac 12,\alpha_p,\alpha_p^{-1},p^{-\frac 12},...,p^{-4a-\frac 32}),
$$
so that the $L$-function is
$$L(s,\pi_f)\prod_{i=0}^n \zeta\Big(s+\frac{n}{2}-i\Big), 
$$
which agrees with the result in Section \ref{Lfunction}. We can also show the matching at the infinity place as in $n$ even case.

\section{Appendix}\label{appendix}
In this section, we explicitly compute $C(S,r)=C(r,\sqrt{\beta}^{-1}\lambda)$ in (\ref{C(S)}) as a generalization of \cite[Theorem A.1]{Po-aut}:
Let 
$$I_v(r,\mu,\lambda)=\int_{\R}t^{r} e^{-t^2}\Bigg(\frac{(t+\lambda \sqrt{-1})^2-\mu}{|(t+\lambda \sqrt{-1})^2-\mu|}\Bigg)^v 
K_v(|(t+\lambda \sqrt{-1})^2-\mu|)dt.
$$
We first consider the case $r$ even.

\begin{theorem}\label{Po-const}
Suppose $\mu,\lambda>0,\ r\in \Z_{\ge 0}$ and $v\in \Z$. Then
$$I_v(2r,\mu,\lambda)=
(-1)^v C(2r,\lambda)e^{-\mu},
$$
where
$$C(2r,\lambda)=\pi e^{\lambda^2}\Gamma(\tfrac 12, \lambda^2) e^{\lambda^2} p_r(\lambda^2)
+\pi e^{-\lambda^2}\lambda q_{r-1}(\lambda^2).
$$
Here $\Gamma(\frac 12,x)=\int_x^\infty t^{-1/2} e^{-t^2}\, dt$ is the incomplete Gamma function, and $p_r,q_{r-1}$ are polynomials of degree $r, r-1$ (resp.) with rational coefficients. If $r=0$, set $q_{-1}=0$.
\end{theorem}

Let $I_v=I_v(r,\mu,\lambda)$. 
 Then by the change of variables, we see that $I_{-v}=I_v$ and $\overline{I_v}=I_v$. So $I_v$ is real.
Also \cite[Proposition A.2]{Po-aut} says, for any $r$,
\begin{equation}\label{recurrence}
\partial_\mu I_v(r,\mu,\lambda)=\frac 12 (I_{v+1}(r,\mu,\lambda)+I_{v-1}(r,\mu,\lambda)).
\end{equation}

Now we explain how to modify the proof of \cite[Proposition A.2]{Po-aut} to obtain $I_0(2r,\mu,\lambda)= C(2r,\lambda)e^{-\mu}.$
Then together with the recurrence relation above, we obtain the theorem.

Let $s=\zeta t^\frac 12+\zeta^{-1} t^{-\frac 12}$, where $\zeta=e^{\sqrt{-1}\frac {\pi}4}$. Then
$s^2=2+\sqrt{-1}(t-t^{-1})$ and $s\bar s=t+t^{-1}$. Let $\alpha=\mu+\lambda^2$. 
As in the proof of \cite[Proposition A.2]{Po-aut}, we have
$$I_0(2r,\mu,\lambda)=e^{-\mu-\lambda^2}\int_0^\infty\int_0^\infty t^{2r} e^{-\frac 12 s^2x^2-\lambda s\bar s x+\frac 12\alpha s^2} \, \frac {dt}t dx.
$$
Now
$$\int_0^\infty t^{2r} e^{-\frac 1{4\beta} x^2-\gamma x}\, dx=4^r\beta^r\left(\sqrt{\beta}\Gamma(r+\tfrac 12){}_1F_1(r+\tfrac 12,\tfrac 12;\beta\gamma^2)-2\beta\gamma \Gamma(r+1){}_1F_1(r+1,\tfrac 32;\beta\gamma^2 \right).
$$
Apply the above formula with $\beta=(s\sqrt{2})^{-2}, \gamma=\lambda s\bar s$. Then 
\begin{eqnarray*}
&& e^\mu I_0(2r,\mu,\lambda)= \\
&& e^{-\lambda^2}\int_0^\infty e^{\frac 12\alpha s^2}\frac {2^\frac n2}{s^{2r}}
\left(\tfrac 1{s\sqrt{2}}\Gamma(r+\tfrac 12){}_1F_1(r+\tfrac 12,\tfrac 12;\tfrac 12\lambda^2{\bar s}^2)-\lambda\tfrac {\bar s}s \Gamma(1+r){}_1F_1(r+1,\tfrac 32; \tfrac 12 \lambda^2{\bar s}^2)\right)\, \frac {dt}t.
\end{eqnarray*}

Let $u=-\sqrt{-1}\frac {\pi}4+\frac 12\log t$ so that $du=\frac {dt}{2t}$, and $\cosh(u)=\frac {\bar s}2$, and $\sinh(u)=-\frac {is}2$. Then
\begin{eqnarray*}
&& e^\mu I_0(2r,\mu,\lambda)=e^{-\lambda^2}\int_{Im(u)=-\sqrt{-1}\frac {\pi}2} e^{-2\alpha (\sinh(u))^2}\frac 1{(-2)^r(\sinh(u))^{2r}} \\
&&\cdot \left(\tfrac 1{2i\sqrt{2}\sinh(u)}\Gamma(r+\tfrac 12){}_1F_1(r+\tfrac 12,2\lambda^2(\cosh(u))^2)-\lambda\tfrac {\cosh(u)}{i\sinh(u)} \Gamma(1+r){}_1F_1(r+1,\tfrac 32; 2\lambda^2 (\cosh(u))^2)\right)\, \frac {dt}t.
\end{eqnarray*}

We differentiate under the integral with respect to $\mu$ and move the contour to $Im(u)=0$. Then the integral becomes an integral of an odd function, and it is zero. Therefore, $I_0(2r,\mu,\lambda)=C(2r,\lambda)e^{-\mu}$ for some $C(2r,\lambda)$. 

In order to compute $C(2r,\lambda)$, consider $\lim_{\mu\to 0} I_0(2r,\mu,\lambda)$. Then
$$C(2r,\lambda)=\int_{-\infty}^\infty t^{2r} e^{-t^2} K_0(t^2+\lambda^2)\, dt.
$$

We use the substitution $t^2+\lambda^2=u$. Then 
$$C(2r,\lambda)=e^{\lambda^2} \int_{\lambda^2}^\infty (u-\lambda^2)^{r-\frac 12} e^{-u}K_0(u)\, du.
$$
We note that $C(0,\lambda)=\frac {\pi}{\sqrt{2}}\Gamma(\frac 12, 2\lambda^2)$.

Let 
$$D(r,\lambda)=e^{\lambda^2} \int_{\lambda^2}^\infty (u-\lambda^2)^{r-\frac 12} e^{-u}K_1(u)\, du.
$$
Then by using the fact that $-\frac {dK_0(t)}{dt}=K_1(t)$, we can show
$$C(2r,\lambda)+D(r,\lambda)=(r-\frac 12)C(2r-2,\lambda),
$$
$$\frac {d}{d\lambda} C(2r,\lambda)=-2\lambda D(r,\lambda).
$$
So we have the differential equation, 
$$
\frac d{d\lambda}C(2r,\lambda)-2\lambda C(2r,\lambda)=2\lambda (r-\frac 12)C(2r-2,\lambda),
$$
Using this recurrence relation, we can show the result.

\begin{remark}\label{rodd}
Suppose $r$ is odd. Then it is clear that $I_0(r,\mu,\lambda)=0$. By the change of variables, we can see that $I_{-v}=-I_v$ and 
$\overline{I_v}=I_{-v}$. Hence $I_v$ is purely imaginary. Let $z=(t+\lambda\sqrt{-1})^2-\mu$. Then
$$I_1(r,\mu,\lambda)=4\lambda\sqrt{-1}\int_0^\infty t^{r+1}e^{-t^2} \frac {K_1(|z|)}{|z|}\, dt.
$$
For $v>1$, $I_v(r,\mu,\lambda)$ can be obtained from $I_1(r,\mu,\lambda)$ by (\ref{recurrence}), 
but we may not have a similar result as in the case when $r$ is even.
\end{remark}

\end{document}